\documentclass[a4paper,10pt]{amsart}
\usepackage[utf8]{inputenc}
\usepackage{amsmath,mathtools,amsthm, amssymb,bm, enumerate}
\usepackage[all]{xy}
\usepackage{xcolor}
\title[Groups of type $\E_8$]{Groups of type $\E_8$ over rings via TKK-algebras and their extremal elements}
\author[S.\ Alsaody and J.\ Desmet]{Seidon Alsaody$^\dagger$ and Jari Desmet$^\ddagger$}
\address{$^\dagger$ Uppsala University, Department of Mathematics, P.\ O.\ Box 480, 751 06 Uppsala, Sweden. Email address: seidon.alsaody@math.uu.se}
\address{$^\ddagger$ Ghent University, Department of Mathematics, Computer Science and Statistics, Krijgslaan 299--S9, 9000 Gent, Belgium.
Email address: jari.desmet@ugent.be}

\newtheorem{Thm}{Theorem}[section]
\newtheorem{Prp}[Thm]{Proposition}
\newtheorem{Cor}[Thm]{Corollary}
\newtheorem{Lma}[Thm]{Lemma} 
\theoremstyle{definition}
\newtheorem{Def}[Thm]{Definition}
\newtheorem{Rk}[Thm]{Remark}
\newtheorem{Ex}[Thm]{Example}
\numberwithin{equation}{section}

\DeclareMathOperator{\Tr}{Tr}
\newcommand{\Ralg}{R\mathchar45\mathbf{rings}}
\newcommand{\Salg}{S\mathchar45\mathbf{rings}}

\renewcommand{\phi}{\varphi}

\newcommand{\ad}{\mathrm{ad}}

\newcommand{\End}{\mathrm{End}}

\newcommand{\C}{\mathbb{C}}

\newcommand{\Z}{\mathbb{Z}}
\newcommand{\Gm}{\mathbb{G}_{\mathrm{m}}}
\newcommand{\Q}{\mathbb{Q}}
\newcommand{\D}{\mathrm{D}}
\newcommand{\E}{\mathrm{E}}
\newcommand{\F}{\mathrm{F}}
\newcommand{\G}{\mathrm{G}}
\newcommand{\tr}{\mathrm{tr}}
\newcommand{\diag}{\mathrm{diag}}
\newcommand{\sll}{\mathfrak{sl}}

\newcommand{\Id}{\mathrm{Id}}

\newcommand{\bAut}{\mathbf{Aut}}
\newcommand{\bSim}{\mathbf{Sim}}

\newcommand{\bAutgr}{\mathbf{Aut}_\mathrm{gr}}

\newcommand{\bGL}{\mathbf{GL}}
\newcommand{\Spec}{\mathrm{Spec}}

\newcommand{\bStab}{\mathbf{Stab}}

\newcommand{\bX}{\mathbf{X}}
\newcommand{\bY}{\mathbf{Y}}
\newcommand{\bG}{\mathbf{G}}

\newcommand{\bC}{\mathbf{C}}
\newcommand{\bU}{\mathbf{U}}
\newcommand{\bE}{\mathbf{E}}

\newcommand{\bInv}{\mathbf{Inv}}
\newcommand{\bStr}{\mathbf{Str}}
\newcommand{\bL}{\mathbf{L}}
\newcommand{\tX}{\widetilde{\mathbf{X}}}
\newcommand{\tY}{\widetilde{\mathbf{Y}}}

\newcommand{\bone}{\mathbf{1}}
\newcommand{\ZS}{\Z[\tfrac16]}
\newcommand{\tbtmat}[4]{\begin{pmatrix}
#1 &#2\\
#3 & #4
\end{pmatrix}}

\begin{document}

\begin{abstract} Over any commutative ring containing $\tfrac16$, we study Lie algebras $L$ of type $\E_8$ that arise from the Tits--Kantor--Koecher (TKK) construction on a Brown algebra, and their twisted forms. We construct a smooth scheme $\bY$ of pairs of extremal elements in $L$. When $L$ arises from the TKK-construction, we express the automorphism group, of type $\E_8$, as an $\E_7$-torsor over $\bY$. We show that twisting by this torsor produces the graded isomorphism classes of those algebras isomorphic to $L$, and parametrize these classes by using $\bY$. We show that this torsor is non-trivial, yielding isomorphic Lie algebras of type $\E_8$ that are not graded isomorphic, as opposed to the behaviour over fields.
\end{abstract}


\maketitle

\section{Introduction}
In the setting of Lie groups, algebraic groups over fields, or affine group schemes over commutative rings, groups of type $\E_8$ are the largest of the exceptional groups, having dimension 248. Their study presents particular difficulties, largely owing to the fact that they do not have non-trivial linear representations of low dimension, a fact that sets them apart from the smaller exceptional groups of types $\G_2$, ($\D_4$), $\F_4$, $\E_6$ and $\E_7$. This is a major reason why groups of type $\E_8$ are seen as ``the most exceptional'', as noted by Garibaldi in \cite{GarE8}, to which we refer for a historical survey and outlook, particularly over fields. 

One way to get around this difficulty is to construct varieties or schemes of lower dimension that arise naturally in some sense, and on which groups of type $\E_8$ act. This is the aim of the present work, set over commutative rings where $6$ is invertible.

A significant class of groups of type $\E_8$ consists of automorphism groups of TKK-algebras, i.e.\ algebras that arise from the TKK-construction, on a Brown algebra. We use this to express the automorphism group of such an algebra $L$ as a fibration (more precisely a homogeneous space or torsor) over a certain scheme $\bY$, with fibres being torsors under groups of type $\E_7$. The scheme $\bY$ consists of pairs of extremal elements in $L$, and part of the work is to view these elements from a scheme theoretic perspective. 

Actions of Lie groups of type $\E_8$ and their stabilizers have been studied in the literature. Notably, \cite{IY} and \cite{YIS} describe, over $\mathbb R$ and $\mathbb C$, the homogeneous space $G/H$ where $G$ is the split or real compact Lie group of type $\E_8$, and $H$ is of type $\E_7$. This is indeed the kind of homogeneous space of interest to us, and the paper \cite{YIS} served as an inspiration for us, even though the setting is quite different. The authors do express real exceptional Lie groups, their Lie algebras and their Killing forms using nonassociative algebras, building on the work \cite{Fre} of Freudenthal. In \cite{GKN}, the authors, again over $\mathbb R$ and $\mathbb C$, study non-linear actions of groups of type $\E_8$ on 57-dimensional linear spaces. This was also a primer in our direction. On the other hand, the general theory of flag varieties of algebraic groups provides a conceptual framework but does not provide an explicit construction in terms of concrete algebras. Our approach is to use the action of $\bAut(L)$ on $\bY$ induced from the natural action on $L$, and describe it geometrically in terms of stabilizers and their fppf-quotients on the one hand, and algebraically in terms of Brown algebras and the TKK-construction. 

The approach follows the programme initiated in \cite{AG}, formalized in \cite{Als1}  and continued in \cite{Als2}, where isotopes of nonassociative algebras are expressed as, and used to parametrize, torsors under smaller groups. The torsors are with respect to the fppf topology, which is the framework used here for quotients, descent and cohomology. In \cite{Als2}, the main objects were Brown algebras and Freudenthal triple systems  constructed from them. The automorphism group of such a Brown algebra has, as connected component, a simple, simply connected group scheme of type $\E_6$, while the invariance group of the related triple system is simple, simply connected of type $\E_7$. This allowed isotopes of these Brown algebras to be used to parametrize homogeneous spaces of type $\E_7/\E_6$. In this work, we use these algebras and their triple systems as our starting point when approaching groups of type $\E_8$ and quotients of type $\E_8/\E_7$. As it turns out, the torsors thus obtained correspond to the different gradings on the TKK-algebras we consider. Various groups play a role, such as structure groups and (graded) automorphism groups, and we prove their smoothness and study their interplay along the way.

Extremal elements in Lie algebras are elements spanning one-dimensional inner ideals. An inner ideal in a Lie algebra $L$ over $R$ is an  $R$-submodule $I\subseteq L$ such that $[I,[I,L]]\subseteq I$. The systematic study of inner ideals in Lie algebras was initiated by Georgia Benkart in \cite{Ben}. Every ideal in a Lie algebra is, \emph{a fortiori}, an inner ideal, but inner ideals exist plentifully also in simple Lie algebras. Extremal elements, i.e.\ elements $x$ satisfying $[x,[x,L]]\in Rx$, have been studied in e.g.\ \cite{CSUW} and \cite{DMM}. They constitute a natural subset in any Lie algebra. Requiring extremal elements to be unimodular removes the singularities, leaving us with a smooth scheme $\bY$ of pairs of extremal elements which will play the key role in our parametrization.

The TKK-construction on Brown algebras over fields has been studied by various authors since Allison in \cite{AllTKK}. Rigby, in his PhD thesis \cite{Rig}, established the theory in a modern language, using algebraic groups. Our presentation is based on his treatment. The theory of Brown algebras rests on that of Albert algebras, for which the standard reference is \cite{GPRbook}.

The paper is organized as follows. Section 2 recalls and lays out the preliminaries of Brown algebras that are necessary for the TKK-construction on Brown algebras, which is presented in Section 3. There, we also give an expression of the associative bilinear form in terms of the TKK-construction, not assuming that $5$ is invertible. In Section 4, we prove the smoothness of the affine group schemes related to these algebras and constructions. Section 5 introduces the functor of extremal elements and establishes its basic properties, while in Section 6 we show that pairs of extremal elements form a smooth affine scheme with a locally transitive $\E_8$-action, where the point stabilizers are of type $\E_7$. The torsors that this action provides are studied in the final Section 7, where their twists are linked to gradings of the TKK-construction, and shown to be in general non-trivial.

\subsection*{Conventions and notation} All rings are assumed to be unital, associative and commutative. We fix a base ring $R$ and throughout, we make the assumption that $6\in R^*$. With the exception of Corollary \ref{Ckillingform}, we do not require $5$ to be invertible. An algebra over $R$ is an $R$-module endowed with a multiplication, which is in general neither associative, commutative or unital. We write multiplication as juxtaposition or as a bracket in the case of Lie algebras.

By an \emph{$R$-ring} we mean a unital, commutative and associative $R$-algebra. We denote the category of $R$-rings by $\Ralg$. Given an $R$-algebra $C$, we denote by $\bAut(C)$ its automorphism group scheme. This is the affine $R$-group scheme defined by $\bAut(C)(S)=\mathrm{Aut}(C\otimes_RS)$ for each $S\in\Ralg$.

Let $M$ be an $R$-module. We denote the $R$-linear span of a subset $X\subseteq M$ by $\langle X\rangle$, and often write $Rx$ for the span of a single element $x\in M$.  If $S$ is an $R$-ring, we write $M_S$ for the $S$-module $M\otimes_RS$. If $\bG$ is a functor defined on $\Ralg$, we write $\bG_S$ for its restriction to $\Salg$. Unadorned tensor products are over $R$.

We call a bilinear form $b$ on $M$ \emph{regular} if the map $x\mapsto b(x,-)$ is an isomorphism from $M$ to its dual.

All sheaves, including all torsors, are with respect to the fppf-topology. If $\bG$ is an $R$-group scheme, we write $H^1_{\mathrm{fppf}}(R,\bG)$ for its first fppf-cohomology; it is a pointed set that is in bijection with the set of $R$-isomorphism classes of $\bG$-torsors over $R$.

\subsection*{Acknowledgment} We are grateful to Michiel Smet for fruitful discussions, which lead to our current version of Proposition \ref{Psmooth} and the reference \cite{Sme}.

\section{Preliminaries: Brown algebras over rings} We will recall the basics of Brown algebras over rings, starting with reduced Brown algebras. For more details we refer to \cite{Als2}. 

Given an Albert algebra $A$ over $R$, set $B=R\oplus R\oplus A\oplus A$, which, following tradition, we write as 
\[B=\begin{pmatrix}
 R&A\\ A& R   
\end{pmatrix}.\]
This is a projective $R$-module of rank 56, which becomes an $R$-algebra with involution under the multiplication
\[\begin{pmatrix}
 r_1&a_1\\ a_2& r_2   
\end{pmatrix}\begin{pmatrix}
 r_1'&a_1'\\ a_2'& r_2'   
\end{pmatrix}=\begin{pmatrix}
 r_1r_1'+T(a_1,a_2')&r_1a_1'+r_2'a_1+a_2\times a_2'\\ r_2a_2'+r_1'a_2+a_1\times a_1'& r_2r_2'+T(a_2,a_1')   
\end{pmatrix},\]
where $T$ and $\times$ denote the trace and cross product of the Albert algebra, respectively, and the involution
\[\begin{pmatrix}
  r_1&a_1\\ a_2& r_2
\end{pmatrix}^*=\begin{pmatrix}
  r_2&a_1\\ a_2& r_1
\end{pmatrix}.\]

We say that $B$ is a \emph{reduced Brown algebra obtained from $A$}. 
If $A$ is the split Albert algebra $A^s_R=A^s_\Z\otimes_\Z R$, then $B$ is the \emph{split Brown algebra} $B^s_R$. By construction, $B^s_R=B^s_{\Z}\otimes_{\Z} R$. We refer to \cite{GPRbook} for the theory of Albert algebras over rings.

More generally, then, a \emph{Brown algebra} over $R$ is an $R$-algebra $B$ with involution $x\mapsto x^*$ whose underlying module is projective of constant rank 56, such that $B\otimes_R S\simeq B^s\otimes_RS$ for some faithfully flat $R$-ring $S$. The algebra $B$ is endowed with a \emph{$V$-operator}
\begin{equation}\label{EV}
    V_{x,y}z=(xy^*)z+(zy^*)x-(zx^*)y.
\end{equation}

Brown algebras are a special case of structurable algebras, defined in \cite{AH} as follows.

\begin{Def}\label{DStructurable} A \emph{structurable algebra} over $R$ is a unital $R$-algebra $B$ with involution such that the commutator relation
\[[V_{x,y},V_{z,w}]=V_{V_{x,y}z,w}-V_{z,V_{y,x}w}\]
holds for all $x,y,z,w\in B$, 
where the $V$-operator is defined by the formula \eqref{EV}.
\end{Def}

(In \cite{AH}, it is noted that the commutator relation need only be required with $y=1$, from which the general case follows.)

\begin{Rk}\label{Rbasechange} Brown algebras are stable under base change: if $B$ is a Brown algebra over $R$ and $R'$ is an $R$-ring, then $B_{R'}=B\otimes_R R'$ is a Brown algebra over $R'$. Indeed, $B_{R'}$ is an algebra with involution that is a projective $R'$-module of constant rank 56. Moreover, if $S$ is a faithfully flat $R$-ring such that $B\otimes_RS$ is isomorphic to the split Brown algebra $B^s_\Z\otimes_\Z S$, then $S'=S\otimes _RR'$ is a faithfully flat $R'$-ring and
    \[B_{R'}\otimes_{R'} S'\simeq (B\otimes_R S)\otimes_SS' \simeq B^s_\Z\otimes_\Z S'\]
    is split. Note that isomorphisms between structurable algebras, by definition, respect the involution.
\end{Rk}

We denote by $B^+$ (resp.\ $B^-$) the submodule of symmetric (resp.\ skew-symmetric) elements, i.e.\ those $x\in B$ with $x^*=x$ (resp.\ $x^*=-x$). Since $2\in R^*$, we have $B^+\cap B^-=\{0\}$, and $x=\frac12(x+x^*)+\frac12(x-x^*)$ for all $x\in B$. This shows that $B=B^+\oplus B^-$, so both summands are finitely generated projective. 

\begin{Rk}\label{Rprodskew}
    It was shown in \cite{Als2} that if $z,w\in B^-$, then $zw$ belongs to $R1_B$, which we identify with $R$.
\end{Rk}

 We will often require $B$ to have a skew element $j$ such that $j^2\in R^*1_B$. The following lemma indicates that this is a natural condition.

\begin{Lma}\label{Lprodskew} Let $j\in B^-$ with $j^2=\mu 1$. Then $\mu$ is invertible if and only if $j$ is conjugate invertible. In that case, the conjugate inverse of $j$ is $-\mu^{-1}j$.
\end{Lma}

Recall that an element $u\in B$ is \emph{conjugate invertible} if there exists $v\in B$ such that $V_{u,v}=\Id_B$. In that case, we know from \cite{AH} that $v$ is unique; it is called the \emph{conjugate inverse} of $u$, denoted $v=\widehat u$, and satisfies $\widehat v=u$.

\begin{proof}
Direct computation shows that if $j\in B^-$, then $V_{j,-j}x=j^2x=\mu x$ for all $x\in B$. If $\mu$ is invertible, this implies that $j$ is conjugate invertible with conjugate inverse  $-\mu^{-1}j$. If, conversely, $j\in B^-$ is conjugate invertible, then by \cite[Proposition 11.1]{AH}, the conjugate inverse $\widehat j$ satisfies
\begin{equation}\label{Eprodskew}
    \widehat j \in B^-, \qquad j\widehat j=\widehat j j=-1, \qquad \text{and}\qquad \forall x\in B: (\widehat j j)x=\widehat j(jx).
\end{equation}
Thus multiplying the equality  $\widehat j j=-1$ from the right by $j$, one has $\widehat j\mu=-j$, and multiplying this by $\widehat j$ one obtains $(\widehat j)^2\mu=-\widehat j j =1$. Thus $\mu$ is invertible.
\end{proof}

It follows that if $j$ is conjugate invertible, then $j(jx)=j^2x=\mu x$ for all $x\in B$. Moreover, the following direct consequence of \cite[Proposition 11.3]{AH} will be useful.

\begin{Lma}\label{Lj} Let $B$ be a Brown algebra and assume $j\in B^-$ satisfies $j^2=\mu1_B$ with $\mu\in R^*$. Then for all $x,y,z\in B$,
\[V_{jx,jy}(jz)=-\mu jV_{x,y}z.\]
\end{Lma}

\begin{Rk}
    In connection to this, we note that \cite{AH} work in the same general setting as we do. In particular, they do not assume that the underlying ring is a field. Therefore, Remark 4.4 and 4.6 from \cite{Als2} were proved over rings already in \cite{AH}. (In \cite{Als2}, it is claimed that \cite{AH} proved them over fields, with an argument that easily generalizes to rings. S.\ Alsaody wishes to apologize for this oversight.)
\end{Rk}

The existence of a conjugate invertible element of $B^-$ is in fact equivalent to $B^-$ being free of rank 1, as shown in the following.

\begin{Prp}\label{Pconjinv}
    If $B^-$ is free of rank 1, then any generator $j$ of $B^-$ is conjugate invertible. Conversely, if $B^-$ contains a conjugate invertible element $j$, then $B^-=Rj$ and hence free of rank 1.
\end{Prp}
\begin{proof}
    Assume that $B^-$ is free and generated by $j$ with $j^2=\mu1_B$. By the above lemma, it suffices to show that $\mu\in R^*$. Suppose $\mu$ is not invertible. Then, by Zorn's Lemma, $\mu$ is contained in some maximal ideal $M\unlhd R$. Since $B$ is a Brown algebra over the ring $R$, $B_K = B\otimes_{R} K$ is a Brown algebra over the field $K = R/M$ by Remark~\ref{Rbasechange}.  Now, $(B_K)^-=(B^-)\otimes_RK$. Since $j$ generates $B^-$ as an $R$-module, $j\otimes 1_K$ spans $(B_K)^-$. Thus any $x\in(B_K)^-$ satisfies $x=j\otimes \lambda$ for some $\lambda\in K$. But then 
    \[x^2=j^2\otimes\lambda^2=\mu\otimes\lambda^2=1_B\otimes\mu\lambda^2=0,\]
    since $\mu\in M$ and $K=R/M$, so in $B_K$, all skew-symmetric elements have square zero. This is a contradiction.

    For the converse, assume that $j\in B^-$ is conjugate invertible with conjugate inverse $\widehat j$, and let $z\in B^-$. Then by \eqref{Eprodskew}, 
    \[z=-(-1)z=-(\widehat j j)z=-\widehat j(jz).\]
    By Remark \ref{Rprodskew}, $jz=\lambda 1_B$ for some $\lambda\in R$, so $z=-\lambda\widehat j$, which by Lemma \ref{Lprodskew} equals $\lambda\mu^{-1}j$. Thus the $R$-module homomorphism $R\to B^-$ defined by $r\mapsto rj$ is surjective. By \eqref{Eprodskew}, it is invertible with inverse $rj\mapsto (-rj)\widehat j$. This completes the proof.
\end{proof}

For the rest of this section, we assume $B^-$ is free and fix a conjugate invertible element $j\in B^-$.  Following \cite{Gar}, we consider the antisymmetric bilinear form 
\[b_j:B\times B\to R, \qquad b_j(x,y)=(xy^*-yx^*)j,\]
and
 the trilinear map
\begin{equation}\label{Etri}
t_j:B\times B\times B\to B, \qquad t_j(x,y,z)=2V_{x,jy}z-b_j(y,z)x-b_j(y,x)z-b_j(x,z)y.\end{equation}

\begin{Rk}\label{RFTS}  From \cite[Section 2.4]{Als2}, we know that $(B,b_j,t_j)$ is a Freudenthal triple system, which we denote by $Q(B,j)$. We refer to \cite{GPR} and \cite{Als2} for the theory of Freudenthal triple systems over rings. When $6\in R^*$, such a triple system is the datum $(M,b,t)$ of an $R$-module $M$ with an alternating bilinear form $b$ and a trilinear map $t:M\times M\times M\to M$, satisfying certain properties. In this paper, we will only be interested in those of the form $Q(B,j)$. 
\end{Rk}

We also define the bilinear form
    \[\beta_j:B\times B\to R, \qquad \beta_j(x,y)=\mu^{-1}b_j(j x,y)=\mu^{-1}((j x)y^*+y(x^*j))j.\]

\begin{Rk}\label{Rbeta} Note that $\beta$ is independent of the choice of $j$. If $j'$ is another conjugate invertible element in $B^-$ with $j'^2=\mu'1_B$, then by Proposition \ref{Pconjinv}, $j'=\lambda j$ and $\mu'=\lambda^2\mu$ for some $\lambda\in R^*$. Then
\[\beta_{j'}(x,y)=\mu'^{-1}\lambda^2((j x)y^*+y(x^*j))j=\mu^{-1}((j x)y^*+y(x^*j))j=\beta_j(x,y).\]
We will therefore drop the subscript $j$. Note that this also implies that $\beta$ is invariant under isomorphism, since $B^-$ is preserved by isomorphisms.
\end{Rk}

\begin{Ex}\label{Ereduced}
Assume that $B$ is reduced, obtained from an Albert algebra $A$.
It comes with a (non-degenerate, antisymmetric) bilinear form $b:B\times B\to R$ defined by
\[b(x,x')=r_1r_2'-r_2r_1'+T(a_1,a_2')-T(a_2,a_1')\]
where
\[\begin{array}{lll}
x=\begin{pmatrix}
 r_1&a_1\\ a_2& r_2   
\end{pmatrix}
& \text{and} &
x'=\begin{pmatrix}
 r_1'&a_1'\\ a_2'& r_2'   
\end{pmatrix}
\end{array},\]
and $T:A\times A\to R$ is the (symmetric, bilinear) trace form of the Albert algebra $A$.

Here, $j=\diag(1,-1)$ satisfies $j^2=1_B$, and $b_j$ coincides with $b$ defined above. In this case, the form $\beta$ satisfies
\[\beta(x,y)=b(jx,y)=r_1r_2'+r_2r_1'+T(a_1,a_2')+T(a_2,a_1').\]
This form is symmetric. It is also regular: by \cite[I.7.1.3]{Knu}, it is enough to check this over any field $K\in\Ralg$, and then by finite dimension, it follows from $\beta$ being injective, which is clear.
\end{Ex}

The fact that $\beta$ is symmetric and regular holds in the general case as well.

\begin{Lma}\label{Lbeta} If $B$ is any Brown algebra containing a invertible skew element, then $\beta$ is a regular, symmetric bilinear form.
\end{Lma}

\begin{proof}The form $\beta$ is bilinear by construction. Symmetry of a bilinear form is invariant under faithfully flat descent, and so is regularity by \cite[I.7.1.3]{Knu}. Since $\beta$ is invariant under isomorphism, it thus suffices to check the reduced case, where both properties were established in Example \ref{Ereduced}.
\end{proof}

\begin{Lma}\label{Lbetaidentity} The bilinear form $\beta$ satisfies
\[\beta(V_{z,w}x,y)=\beta(V_{x,y}z,w)=\beta(V_{y,x}w,z).\]
\end{Lma}

\begin{proof} It is enough to prove the equalities after a faithfully flat base change. Thus we may assume that $B$ is split and, by Remark \ref{Rbeta}, that $j=\diag(1,-1)$. First we prove the auxiliary equalities
\[b_j(y,V_{z,jw}x)=b_j(w,V_{x,jy}z)=b_j(z,V_{y,jx}w).\]
Indeed, using \eqref{Etri}, the middle expression equals
\[b(w,t_j(x,y,z))+b_j(y,z)b_j(w,x)+b_j(y,x)b_j(w,z)+b_j(x,z)b_j(w,y).\]
From this, the left-hand side is obtained via the swaps $x\leftrightarrow z$ and $y\leftrightarrow w$, and the right-hand side is obtained via the swaps $x\leftrightarrow y$ and $z\leftrightarrow w$. By \cite[Remark 2.7]{Als2}, the first term is symmetric, hence invariant under each of these pairs of swaps. The three remaining terms are also unchanged, as $b$ is antisymmetric. This proves the auxiliary equalities.

The left-most equality in the lemma is obtained from the left-most auxiliary equality by replacing $y$ by $jy$ and $w$ by $jw$, and using the symmetry of $\beta$. To obtain the right-most equality, we first replace $w$ by $jw$ and $y$ by $jy$ in the right-most auxiliary equality, which gives
\[b_j(jw,V_{x,y}z)=b_j(z,V_{jy,jx}jw).\]
The left-hand side is precisely $\beta(V_{x,y}z,w)$, and by Lemma \ref{Lj}, the right-hand side equals
\[b_j(z,-jV_{y,x}w)=b_j(jV_{y,x}w,z)=\beta(V_{y,x}w,z),\]
which proves the right-most equality of the lemma.
\end{proof}

\section{TKK-algebras}
In this section, we will introduce the main class of algebras we consider in this paper. This is a class of $\Z$-graded algebras over $R$. We therefore begin with some necessary generalities on graded algebras.

\subsection{Preliminaries on $\Z$-graded algebras}
An $R$-algebra $L$ is \emph{$\Z$-graded} if $L$ decomposes as a direct sum of $R$-submodules $L=\bigoplus_{i\in \Z} L_i$, such that $L_iL_j\subseteq L_{i+j}$ for all $i,j\in\Z$. This property is preserved by arbitrary base change: the scalar extension $L_S=L\otimes_RS$ is $\Z$-graded for any $R$-ring $S$. Namely, the grading on $L$ induces a grading on $L_S$ by setting $(L_S)_i=(L_i)_S$, since $\bigoplus_i (L_i\otimes S)=L\otimes S$ by the distributivity of the tensor product over arbitrary direct sums.

 We will use two ways to characterize the graded components. For the first, we follow \cite{Rig} and \cite{Sta} in considering the \emph{grading cocharacter}  $\chi$ associated to the grading. If $L=\bigoplus_{i\in\Z} L_i$ is a $\Z$-graded algebra, this is the $R$-group homomorphism
\[\chi:\Gm\to\bAut(L)\]
where for each $R$-ring $S$ and $s\in \Gm(S)=S^*$, the automorphism $\chi_S(s)$ of $L_S$ is defined by $x\mapsto s^ix$ for each $x\in (L_S)_i$. 

If the graded components are finitely generated projective, the grading cocharacter determines the grading in the following way.

\begin{Lma}\label{Lcharacter} Let $L=\bigoplus_i L_i$ be a $\Z$-graded $R$-algebra, where each $L_i$ is finitely generated projective. Let $\chi:\Gm\to\bAut(L)$ be the grading cocharacter and set
\[L_i^\chi=\{x\in L\mid \forall S\in\Ralg, \forall s\in S^*: \chi_S(s)(x\otimes 1_S)=s^i (x\otimes 1_S)\}.\]
Then $L_i^\chi=L_i$ for each $i\in\Z$.
\end{Lma} 

To alleviate the notation, we use the shorthand $x_S:=x\otimes 1_S$.

\begin{proof}
    If $x\in L_i$, then $x_S\in (L_S)_i$, so that $\chi_S(s)(x_S)=s^i(x_S)$ for all $s\in S^*$, where $S$ is any $R$-ring. Thus $L_i\subseteq L_i^\chi$. To show that the inclusion is an equality, take $x\in L_i^\chi$. It can be written uniquely as $x=\Sigma_j x_j$ with $x_j\in L_j$, and our goal is to prove that $x_j=0$ whenever $j\neq i$. If $S$ is an $R$-ring and $s\in S^*$, then on the one hand
\[\chi_S(s)(x_S)=s^i x_S=\textstyle\sum_j s^i (x_j)_S\]
while on the other
\[\chi_S(s)(x_S)=\textstyle\sum_j \chi_S(s)(x_j)_S=\sum_j s^j (x_j)_S,\]
so that $(s^i-s^j)(x_j)_S=0$ in $(L_S)_j$ for all $j$. In particular, this holds with $S=R[t,t^{-1}]$, the Laurent polynomial ring over $R$, and $s=t$. If $i\neq j$ and $x_j\neq 0$, then $(x_j)_S\neq 0$ since $S$ is free as an $R$-module. Since $(L_j)_S$ is projective, hence torsion free, this implies that $t^i-t^j$ is a zero divisor. Thus $(t^i-t^j)p(t)=0$ for some $p(t)\in R[t,t^{-1}]$. Multiplying by $t^{m+n}$ with $m, n$ such that $m+i$ and $m+j$ are non-negative and $t^np(t)\in R[t]$, we see that $t^{m+i}-t^{m+j}$ is a zero divisor in the polynomial ring $R[t]$. But this is impossible by McCoy's theorem, since the coefficients are units. This concludes the proof.
\end{proof}

A linearized version of the above lemma makes use of certain derivations. Following the terminology of  \cite{Sta}, if $L=\bigoplus_{i\in\Z} L_i$ is a $\Z$-graded algebra, the associated \emph{grading derivation} $D:L\to L$ is the $R$-linear map defined by $D(x)=ix$ for each $x\in L_i$. It is a derivation as it satisfies the Leibniz rule, and it determines the grading in the following case (which will be the case of interest for us).

\begin{Lma}\label{Lderivation} Let $L=\bigoplus_i L_i$ be a $\Z$-graded $R$-algebra, such that $L_i=0$ whenever $|i|>2$. Let $D:L\to L$ be the grading derivation and set
\[L_i^D=\{x\in L\mid D(x)=ix\}.\]
Then $L_i^D=L_i$ for each $i\in\Z$.
\end{Lma} 

\begin{proof}
By definition, $L_i\subseteq L_i^D$ for all $i$. Conversely, take $x\in L_i^D$ and write $x=\Sigma_j x_j$ with $x_j\in L_j$. Then
\[\textstyle\sum_j ix_j=ix= D(x)=\sum_jD(x_j)=\sum_j jx_j,\]
so that $(j-i)x_j=0$ for all $j$. By assumption, $x_j=0$ if $|j|>2$, and since $6\in R^*$, we have $j-i\in R^*$ whenever $0\neq |j-i|\leq 2$. Thus $x_j=0$ unless $i=j$, so $x\in L_i$, which completes the proof.
\end{proof}

\begin{Rk}\label{Rderivation} We will mainly be dealing with the situation where $L$ is a Lie algebra and $D=\ad_z$ is an inner derivation. Such an element $z$ will be called a \emph{grading element}. (In \cite{Sta}, the element itself is also called a grading derivation.) If $\phi\in\bAut(L)(R)$ is an automorphism of $L$ fixing $z$, then $\phi$ is a graded automorphism of $L$: indeed, if $x\in L_i$, then
\[D\phi(x)=[z,\phi(x)]=[\phi(z),\phi(x)]=\phi D(x)=\phi(ix)=i\phi(x),\]
so that $\phi(x)\in L_i^D$, which is equal to $L_i$ by the above lemma.    
\end{Rk}

\subsection{The TKK-construction}

Let $B$ be a structurable algebra over $R$. Following \cite{AllTKK} (see also \cite{Rig}), we will define a $\Z$-graded algebra $L=L(B)$ as follows. Writing $L_i$ for the homogeneous component of degree $i\in\Z$, we let, as $R$-modules, 
\begin{itemize}
    \item $L_i=0$ for $i<-2$ and $i>2$,
    \item each of $L_{\pm2}$ be a copy of $B^-$, 
    \item each of $L_{\pm1}$ be a copy of $B$, and
    \item $L_0$ be the \emph{inner structure Lie algebra} $\mathfrak{instrl}(B)$, i.e.\ the $R$-linear span of all $V$-operators on $B$, which is closed under the commutator in view of the identity
\[[V_{x,y},V_{z,w}]=V_{V_{x,y}z,w}-V_{z,V_{y,x}w}\]
    from Definition \ref{DStructurable}. Note that since $\lambda V_{x,y}=V_{\lambda x,y}$ for all $\lambda\in R$, each element in $L_0$ is a sum of $V$-operators.
\end{itemize}

The Lie bracket is the unique antisymmetric, bilinear map $[,]:L\times L\to L$ extending the Lie commutator on $L_0$ by

\begin{tabular}{ll}
$L_0\times L_1\to L_1$, & $[V_{x,y},z_+]=V_{x,y}(z_+)$\\
    $L_0\times L_{-1}\to L_{-1}$, & $[V_{x,y},z_-]=-V_{y,x}(z_-)$\\
     $L_0\times L_2\to L_2$, &  $[V_{x,y},s_+]=(s_+y)x^*+x(y^*s_+)$\\
     $L_0\times L_{-2}\to L_{-2}$, &$[V_{x,y},s_-]=-(s_-x)y^*-y(x^*s_-)$\\ 
     $L_1\times L_1\to L_2$, & $[x_+,y_+]=x_+y_+^*-y_+x_+^*$\\    
     $L_{-1}\times L_{-1}\to L_{-2}$, &$[x_-,y_-]=x_-y_-^*-y_-x_-^*$\\ 
     $L_1\times L_{-1}\to L_0$, &$[x_+,y_-]=V_{x_+,y_-}$\\
    $L_2\times L_{-2}\to L_0$, &$[s_+,t_-]=V_{s_+t_-,1}$\\
     $L_2\times L_{-1}\to L_1$, & $[s_+,x_-]=s_+x_-$\\
     $L_{-2}\times L_{1}\to L_{-1}$, &$[s_-,x_+]=s_-x_+$\\
\end{tabular}

\vspace{1ex}

We will mostly be interested in the case where $B$ is a Brown algebra. 

\begin{Rk}
If $B$ is a Brown algebra such that $B^-$ contains $j$ with $j^2=\mu\in R^*$, the multiplications into $L_{\pm2}$ can be written more simply as

\begin{tabular}{lll}
     $L_0\times L_{\pm2}\to L_{\pm2}$, &  $[\Sigma_iV_{x_i,y_i},sj]=\pm s\Sigma_i\beta(x_i,y_i)j$&\text{and}\\ 
     $L_{\pm1}\times L_{\pm1}\to L_{\pm2}$, & $[x_\pm,y_\pm]=\mu^{-1}b(x_\pm,y_\pm)j$\\
\end{tabular}

\end{Rk}
 \begin{Rk}\label{Rbasechange2} By construction, any isomorphism $B\overset{\sim}{\longrightarrow}B'$ induces a graded isomorphism $L(B)\overset{\sim}{\longrightarrow} L(B')$. Moreover, the construction is compatible with base change in the sense that $L(B\otimes_RS)=L(B)\otimes_R S$ for any $S\in\Ralg$. 
\end{Rk}

When $B$ is the split Brown algebra over $R$, we call $L(B)$ the \emph{split TKK-algebra of Brown type $\E_8$}, and denote it by $L^s_R$. For any Brown algebra $B$, we say that $L(B)$ is \emph{obtained from $B$ by the TKK-construction}. Before defining the general notion of a TKK-algebra of Brown type $\E_8$, we show that $L(B)$ is a Lie algebra of type $\E_8$. 

To begin with, the following lemmata show that the underlying module is finitely generated projective. Since $B$ and $B^-$ enjoy this property, what is left to check is that the inner structure algebra $\mathfrak{instrl}(B) =\langle V_{x,y} \mid  x,y \in B\rangle$ is projective as an $R$-module, and has constant rank.  If $A$ is an Albert algebra and $a,b\in A$, we recall the definition of the $V$-operator $V^A_{a,b}\in \End(A)$ by $V^A_{a,b}(c) = T(a,b)c+T(a,c) - (c\times a)\times b$. (We use the superscript $A$ to distinguish it from the $V$-operator on $B$.)
\begin{Lma}\label{lem:Vop}
    Let $B$ be a reduced Brown algebra over $R$, obtained from an Albert algebra $A$. Let 
    \[x,y \in \begin{pmatrix}
        R & 0 \\
        0 & 0
    \end{pmatrix}\cup \begin{pmatrix}
        0 & A \\
        0 & 0
    \end{pmatrix}\cup \begin{pmatrix}
        0 & 0 \\
        A & 0
    \end{pmatrix}\cup \begin{pmatrix}
        0 & 0 \\
        0 & R
    \end{pmatrix}.\]
    Then $V_{x,y}\in V_{B,1}$ unless $\{x,y\} = \left\{\tbtmat{0}{a}{0}{0},\tbtmat{0}{0}{b}{0}\right\}$ for $a,b\in A$. More precisely, in those cases we have
    \[V_{\tbtmat{0}{a}{0}{0},\tbtmat{0}{0}{b}{0}} \begin{pmatrix}
			t & c \\ d & u
		\end{pmatrix} = \begin{pmatrix}
		0 & V^A_{a,b} c \\ T(a,b)d - V^A_{b,a} d & T(a,b) u
	\end{pmatrix},\]
    and
    \[V_{\tbtmat{0}{0}{b}{0},\tbtmat{0}{a}{0}{0}} \begin{pmatrix}
			t & c \\ d & u 
		\end{pmatrix} = \begin{pmatrix}
		T(a,b) t & T(a,b)c - V^A_{a,b} c \\ V^A_{b,a}d & 0
	\end{pmatrix}.\]
\end{Lma}
\begin{proof}
    The computations for this proof have been done for \cite[Lemma~3.2.3]{Sme} in the context of cubic norm pairs.
\end{proof}
\begin{Lma}\label{Lproj}
Let $B$ be the split Brown algebra over $R$, obtained from the split Albert algebra $A$. Then
\begin{enumerate}[(i)]
    \item As $R$-modules, we have\[\mathfrak{instrl}(B) \cong B \oplus \mathfrak{inder}(B),\]
where $\mathfrak{inder}(B) = \{D\in \mathfrak{instrl}(B) \mid D(1) =0 \}$.
    \item We have $\mathfrak{inder}(B) \cong  A^0 \oplus \mathfrak{der}(A)$, where $A^0 = \{ a\in A \mid T(a,1) = 0 \}$. 
\end{enumerate}
As a consequence, $\mathfrak{instrl}(B)$ is projective of constant rank $134$.
\end{Lma}
\begin{proof}
    The first part follows from the fact that $\mathfrak{inder}(B)$ is the kernel of the map $\mathfrak{instrl}(B) \to B \colon V \mapsto V(1)$, while $V_{B,1} \to B \colon V_{x,1} \mapsto V_{x,1}(1) = 2x-x^*$ is an isomorphim of $R$-modules because $3\in R^*$.

    By Lemma~\ref{lem:Vop}, we have 
    \[\mathfrak{inder}(B) \cong  W = \langle V^A_{a,b}-V^A_{b,a}, V^A_{c,d} \mid a,b,c,d\in A, T(c,d) = 0\rangle \subseteq\End(A).\]
    In \cite[Proposition~51.29]{GPRbook}, it was proven that $\mathfrak{der}(A)= \langle V^A_{a,b}-V^A_{b,a} \rangle$, where $\mathfrak{der}(A)$ is defined as in \cite[50.3]{GPRbook}. On the other hand, for $c,d\in A$ with $T(c,d) = 0$, we have $T(V^A_{c,d}1,1) =0 $ and 
    \[V^A_{c,d}-\frac{1}{2}V^A_{V^A_{c,d}1,1}\in \mathfrak{der}(A)\]
    since applying it to $1\in A$ gives zero. Hence $W \cong V^A_{A^0,1}\oplus \mathfrak{der}(A)$, and as the linear map $a\mapsto V^A_{a,1}$ maps $A^0$ bijectively to $V_{A^0,1}$, this proves (ii).

    The last assertion now follows from \cite[Proposition~51.29]{GPRbook}, where it is proved that $\mathfrak{der}(A)$ is projective of constant rank 52, while $A^0$ and $B$ are projective of constant rank 26 and 56, respectively.
\end{proof}

\begin{Prp}\label{PTKK} Let $B$ be a Brown algebra over $R$. The algebra $L=L(B)$ is a graded Lie algebra, which is finitely generated projective as an $R$-module. For each $R$-field $K$, $L_K$ is a simple $K$-Lie algebra of type $\mathrm E_8$.
\end{Prp}

\begin{proof} Assume first that $B$ is split. Then by the above lemma, $L(B)$ is finitely generated projective as an $R$-module. Further, $B=B^s\otimes_{\ZS} R$, where $B^s$ is the split Brown algebra over $\ZS$. Therefore, by the Remark \ref{Rbasechange2}, $L(B)=L(B^s)\otimes_{\ZS} R$. Since the grading and antisymmetry hold by definition, to show that $L(B)$ is a Lie algebra, we need to show the Jacobi identity. This follows by base change from the Jacobi identity over $\ZS$, so we may assume that $R=\ZS$. Since $L(B)$ is a projective, hence free,  $\ZS$-module, and since we have an injection $\ZS\to\Q$, an element $x\in L(B)$ is zero if and only if it becomes zero after base change to $\Q$. This reduces the statement to the case where $R=K$ is a field. Then $L_K=L(B^s\otimes_{\ZS}K)$, and by \cite[5.3]{Rig}, this is a simple Lie algebra over $K$ of type $\E_8$.

In the general case, $L(B)$ is fppf-locally isomorphic, as a graded algebra, to $L^s_R=L(B^s_R)$. By descent, it is thus projective as an $R$-module, and the multiplication is antisymmetric and satisfies the Jacobi identity. Thus $L(B)$ is a Lie algebra, and \cite[5.3]{Rig} gives its type.
\end{proof} 

For the general definition of a TKK-algebra of Brown type $\E_8$, we define the numbers $d_i, i\in\Z$, by
\[d_i=\left\{\begin{array}{ll}
    134 & \text{if\ } i=0, \\
    56 & \text{if\ } i=\pm1,\\
    1 & \text{if\ } i=\pm2, \text{\ and}\\
        0 & \text{otherwise}.
\end{array}\right.\]

\begin{Def}\label{DTKK} A \emph{TKK-algebra of Brown type $\E_8$} over $R$ is a $\Z$-graded algebra $L=\bigoplus_{i\in\Z} L_i$ over $R$, such that $L_i$ is finitely generated projective of constant rank $d_i$, and such that $L$ is fppf-locally isomorphic, as a graded algebra, to $L^s_R$. 
\end{Def}

We will call such a grading a \emph{TKK-grading of Brown type $\E_8$}.

Thus a TKK-algebra of type $\E_8$ is a graded algebra $L$ such that, for some faithfully flat $R$-ring $S$, there is an isomorphism $\phi: L\otimes_RS\to L^s_S$ such that $\phi((L\otimes_RS)_i)\subseteq (L^s_S)_i$ for each $i\in \Z$. By faithfully flat descent, every TKK-algebra of Brown type $\E_8$ is a Lie algebra. In the following, we will omit the words ``of Brown type $\E_8$'' as we will not consider other types.

Proposition \ref{PTKK} has the following immediate consequence.

\begin{Cor}
    For any Brown algebra $B$, the Lie algebra $L(B)$ is a TKK-algebra.
\end{Cor}

\begin{Rk} Let $B$ be a Brown algebra over $R$, and set $z=V_{1,1}\in L_0$, then by the definition of the multiplication, $\ad_z$ is the grading derivation associated with the grading in the construction above. Thus $z$ is a grading element of that grading.
\end{Rk}

The following proposition clarifies the different properties that a TKK-algebra can have.

\begin{Prp} Let $L$ be a $\Z$-graded algebra over $R$. Consider the following statements.
\begin{enumerate}[(i)]
    \item $L$ is the split TKK-algebra.
    \item $L=L(B)$ where $B$ is the split Brown algebra.
    \item $L=L(B)$ where $B$ is a reduced Brown algebra.
 \item $L=L(B)$ where $B$ is a Brown algebra such that $B^-$ is free.
    \item $L=L(B)$ where $B$ is a Brown algebra containing $j\in B^-$ with $j^2\in R^*1_B$.
    \item $L=L(B)$ where $B$ is a Brown algebra.
    \item $L$ is a TKK-algebra.
\end{enumerate}
Then $(i)\Leftrightarrow(ii)\Rightarrow(iii)\Rightarrow(iv)\Leftrightarrow (v) \Rightarrow(vi)\Rightarrow(vii)$. If $R$ is a local ring, then $(iv)\Leftrightarrow(vi)$. If $R$ is an algebraically closed field, all statements are equivalent.
\end{Prp}

\begin{proof}
The left-to-right implications, except possibly $(iv)\Rightarrow(v)$, hold by definition; as does the implication $(ii)\Rightarrow(i)$. The equivalence $(iv)\Leftrightarrow(v)$ holds by Proposition \ref{Pconjinv}. When $R$ is a local ring, every finitely generated projective module is free of finite rank, and when $R$ is an algebraically closed field, every simple Lie algebra of type $\E_8$ is isomorphic to the split one. 
\end{proof}

\subsection{Associative form}
We will need an associative bilinear form on $L=L(B)$. Note that as $L$ is finitely generated projective, one can define a trace form $\Tr$ on $\End_R(L)$, see e.g.\ \cite[4.3]{Bou}. Using that trace form, the Killing form is defined by $\kappa(x,y) = \Tr(\ad_x\ad_y)$. By \cite[4.3(18-20)]{Bou} and the Jacobi identity, it satisfies the usual properties of a trace form, i.e., it is an associative, symmetric bilinear form.

However, since we do not require 5 to be invertible, this form will not always be useful. Rather, we will define another bilinear form explicitly. Later we will see that it is closely connected to the extremal elements we will study below, and prove that it has the desired properties without requiring $5\in R^*$. Over fields of good characteristic, i.e.\ where 30 is invertible, our form is a scalar multiple of the Killing form. Therefore, we obtain a simple expression for the Killing form in terms of the TKK-algebra structure.

\begin{Def}\label{Dtau}
Let $L=L(B)$ and write $x=(s_{-},x_{-},\Sigma_iV_{a_i,b_i},x_+,s_+)$ and  $y=(t_{-},y_{-},\Sigma_jV_{c_j,d_j},y_+,t_+)$ for elements of $L$. We define $\tau:L\times L\to R$ by 
\[\tau(x,y)=\Sigma_{i,j}\beta(V_{a_i,b_i}c_j, d_j)+\beta(x_+,y_{-})+\beta(x_-,y_+)+s_+t_{-}+s_{-}t_+.\]
We call $\tau$ \emph{the bilinear form of $L$}.
\end{Def}

Here $\beta$ is the symmetric bilinear form on $B$ from Lemma \ref{Lbeta}. Since $V$-operators are linear, $\tau$ is independent of the representation of the 0-component of $x$ as a sum of $V$-operators. This also holds for the $0$-component of $y$, since by Lemma \ref{Lbetaidentity}, the first term is equal to $\Sigma_{i,j}\beta(V_{c_j,d_j}a_i, b_i)$. Moreover, by Remark \ref{Rprodskew}, $s_-t_+$ and $s_+t_-$ are elements of $R1_B$, which we identify with $R$. Thus the form $\tau$ is well-defined.

\begin{Prp}\label{Pregular} The form $\tau$ is a regular, symmetric bilinear form. Moreover, it is associative in the sense that, for all $x,y,z\in L$,
\[\tau([x,y], z)=\tau(x,[y,z]).\]    
\end{Prp}

\begin{proof} The form $\tau$ is bilinear by construction. By \cite[I.7.1.3]{Knu}, $\tau$ is regular if and only if $\tau_K$ is regular for each $R$-field $K$, which in turn is equivalent to the map $x\mapsto\tau_K(x,-)$ being injective, since $L_K$ is a finite-dimensional $K$-vector space. If $x=(s_{-},x_{-},\Sigma_iV_{a_i,b_i},x_+,s_+)\in L_K$ is such that $\tau_K(x,-)$ is the zero map, then from the definition of $\tau$ and the regularity of  $\beta$, it follows that $s_\pm=0$ and $x_\pm=0$, and that $\Sigma_iV_{a_i,b_i}c=0$ for all $c\in B$. Thus $x=0$, proving that $\tau$ is regular. As for symmetry, the only non-trivial part is the symmetry of the first term, which follows from Lemma \ref{Lbetaidentity}. 

To prove associativity, it suffices to do so after a faithfully flat extension, so we may assume that $B$ is split and that $j=\diag(\lambda,-\lambda)$ for some $\lambda\in R^*$. The identity being linear in each variable, it suffices to consider the cases where $x$, $y$ and $z$ are homogeneous. In that case, $\tau([x,y],z)=\tau(x,[y,z])=0$ unless $\deg x+\deg y+\deg z=0$.
This leaves 19 non-trivial possibilities for the triple $(\deg x,\deg y,\deg z)$. By symmetry of $\tau$ and antisymmetry of the Lie bracket, we have the equivalence
\[\tau([x,y], z)=\tau(x,[y,z]) \Longleftrightarrow \tau([z,y], x)=\tau(z,[y,x]).\]
Thus for each value of $\deg y$, it suffices to consider the triples $(\deg x,\deg y,\deg z)$ where, say, $\deg x\leq\deg z$. It thus suffices to check the identity in the following 11 possibilities for $(\deg x,\deg y,\deg z)$:
\[\begin{array}{cccccc}
  (-2,0,2)   & (-2,1,1) & (-2,2,0) & (-1,-1,2) & (-1,0,1) & (-1,1,0)   \\
    (-1,2,-1) & (0,-2,2) & (0,-1,1) & (0,0,0) & \text{and} & (1,-2,1). 
\end{array}\]
Moreover, in degree 0, bilinearity reduces this to the case of a single $V$-operator.

In the case $(0,0,0)$, setting $x=V_{a,b}$, $y=V_{c,d}$ and $z=V_{e,f}$, and using the formula for the commutator of $V$-operators together with Lemma \ref{Lbetaidentity}, we have
\[\begin{array}{rl}
     \tau([x,y],z)&=\beta(V_{V_{a,b}c,d}e,f)-\beta(V_{c,V_{b,a}d}e,f) \\
     &=\beta(V_{e,f}V_{a,b}c,d)-\beta(V_{e,f}c,V_{b,a}d)\\
     &=\beta(V_{f,e}d,V_{a,b}c)-\beta(V_{b,a}d,V_{e,f}c)\\
     &=\beta(V_{a,b}c,V_{f,e}d)-\beta(V_{a,b}V_{e,f}c,d)\\
     &=\beta(V_{c,V_{f,e}d}a,b)-\beta(V_{V_{e,f}c,d}a,b)\\
     &=-\tau([z,y],x)=\tau(x,[y,z]).
\end{array}\]

In the case $(-1,0,1)$, setting $x=x_-\in L_{-1}$, $y=V_{a,b}$ and $z=z_+\in L_1$, we compute
\[\tau([x,y],z)=\beta(V_{b,a}x_-,z_+)=\beta(V_{a,b}z_+,x_-)=\tau(x,[y,z]),\]
by Lemma \ref{Lbetaidentity}. The case $(0,-1,1)$ is similar. In the case $(-1,1,0)$, setting $x=x_-\in L_{-1}$, $y=y_+\in L_{1}$ and $z=V_{a,b}$, the same lemma gives
\[\tau([x,y],z)=-\beta(V_{y_+,x_-}a,b)=-\beta(V_{a,b}y_+,x_-)=\tau(x,[y,z]).\]
The other cases are proved by straight-forward computations.
\end{proof}

\begin{Cor}\label{Ckillingform} If $R$ is a field of characteristic not dividing 30, then $\kappa=60\tau$ is the Killing form of $L$.
\end{Cor}

\begin{proof} We may assume that $R$ is algebraically closed and deduce the general case by descent. From \cite{Rig}, we know that $L$ is a simple Lie algebra. Thus any non-degenerate, associative, symmetric, bilinear form is a scalar multiple of the Killing form $\kappa$. Thus $\kappa=\lambda\tau$ for some $\lambda\in R$. To determine $\lambda$, we set $x=V_{1,1}$. On the one hand, the above formula gives \[\tau(x,x)=\beta(1,1)=2\]
while on the other hand, since $x$ is the grading element, $\ad_x$ acts as $i\Id$ on the graded component $L_i$, so
\[\tr(\ad_x\ad_x)=2^2+56+0+56+2^2=120.\]
The claim follows.
\end{proof}

\subsection{Automorphisms and the bilinear form}
Let $L=L(B)$. We will show that $\bAut(L)$ preserves the bilinear form $\tau$ of $L$. While we will not need this result later on, we include it for completeness. To this end, let $V$ be the $R$-module of bilinear forms on $L$. The natural action of $\bG=\bAut(L)$ on $L$ extends to an action of $\bG$ on $V$, making it a $\bG$-module: if $S\in\Ralg$, $\phi\in\bG(S)$ and $\sigma\in V_S$, then
\[\phi\cdot \sigma: (x,y)\mapsto \sigma(\phi(x),\phi(y)).\]

For each $S\in\Ralg$, let $(V\otimes_RS)^{\bG_S}$ be the set of all $\sigma\in V_S$ that are fixed by the action of the the $S$-group scheme $\bG_{S}$. Let moreover $\bG^\tau$ be the subgroup scheme of $\bG$ stabilizing $\tau$.

\begin{Prp}\label{Ptaustable} Let $L=L(B)$ with bilinear form $\tau$. Then $\bG^\tau=\bG$.
\end{Prp}

\begin{proof} First we treat the case where $R=\ZS$ and  $L(B)$ is the split TKK-algebra of type $\E_8$ over $R$. Let $k=\overline{\Q}$. Then $\tau_{k}$ is a scalar multiple of the Killing form over $k$, by Corollary \ref{Ckillingform}. Therefore it is fixed by $\bG_{k}$, i.e.\ $\tau_{k}\in (V_{k})^{\bG_{k}}$. Since $R\to k$ is a flat ring extension, from \cite[Lemma 2]{Ses} we thus have
\[\tau_{k}\in(V\otimes_Rk)^{\bG_{k}}=(V^\bG)\otimes_Rk,\]
so that $\tau\in V^\bG$. In other words, the inclusion $\bG^\tau\to \bG$ is an isomorphism in the split case over $\ZS$. It is therefore \emph{a fortiori} an isomorphism in the split case and, by faithfully flat descent, in the general case, over any ring $R$. Then for each $R$-ring $S$, the inclusion $\bG^\tau(S)\to \bG(S)$ is an isomorphism of groups, hence an equality. The result follows.\end{proof}

\section{Related groups}
Let $L$ be a TKK-algebra over $R$. As above, we denote by $\bG$ the automorphism group scheme $\bAut(L)$ of $L$. It is an affine group scheme over $R$.

\begin{Prp}\label{Smoothaut} The group $\bG=\bAut(L)$ is a semisimple (hence in particular smooth) affine group scheme of type $\E_8$.
\end{Prp}

\begin{proof} Since $L$ is fppf-locally isomorphic to the split algebra $L^s_R$, the group $\bG$ is an fppf-twisted form of $\bAut(L^s_R)$. Since $L^s_R$ is obtained by base change from the split TKK-algebra over $\Z[\frac16]$, we may assume that $L$ is split and that $R=\Z[\frac16]$, which is a PID and hence a Dedekind domain.

 The group $\bG$ is finitely presented. By \cite[7.1 and 11.2]{Rig},  for any field $K$ with $6\in K^*$, $\bG_K$ is a semisimple group of type $\E_8$, hence in particular smooth, connected and of dimension 248. From \cite[Lemma B.1]{AG} it thus follows that $\bG$ is smooth and connected. Since semisimplicity and type are defined fibrewise, the claim follows. 
\end{proof}

\subsection{Graded automorphisms}

We will also be interested in the graded automorphism group $\bAutgr(L)$. This is defined for any $\Z$-graded $R$-algebra $L$ by
\[\bAutgr(L)(S)=\{\phi\in\bAut(L)(S)\mid \forall i\in\Z: \phi(L_i\otimes S)=L_i\otimes S \}\]
 for each $S\in \Ralg$. Let $\chi$ be the grading cocharacter of $L$. Recall that the centralizer $\bC_\bG(\chi)$ is defined by
\[\bC_\bG(\chi)(S)=\{g\in\bG(S)\mid \forall T\in \Salg, \forall t\in T^*:  g_T\chi(t)g_T^{-1}=\chi(t) \}\]
for any $R$-ring $S$. The following general lemma is helpful.

\begin{Lma} Let $L$ be a $\Z$-graded algebra over $R$, and denote the grading character by $\chi$. If $\bG=\bAut(L)$ is a smooth affine $R$-group scheme, then $\bC_\bG(\chi)$ is a smooth closed subgroup scheme of $\bG$.  Moreover, the inclusion $\bC_\bG(\chi)\to\bG$ factors through a homomorphism $\bC_\bG(\chi)\to\bAutgr(L)$.
\end{Lma}

\begin{proof} Since $\Gm$ is of multiplicative type and finitely generated, the first statement follows from \cite[Expos\'e XI, 5.3]{SGA3} (see also \cite[Lemma 2.2.4]{Con}). As for the second statement, let $S$ be an $R$-ring and $\phi\in \bC_\bG(\chi)(S)$. We need to show that $\phi(x)\in (L_S)_i$ for every $x\in (L_S)_i$. By Lemma \ref{Lcharacter}, this amounts to showing $\phi(x)\in (L_S)_i^\chi$. If $T$ is any $S$-ring and $t\in T^*$, then by definition of $\bC_\bG(\chi)$,
\[\chi(t)\phi_T(x_T)=\phi_T\chi(t)(x_T)=\phi_T(t^ix)=t^i\phi_T(x_T).\]
Since $\phi_T(x_T)=\phi(x)_T$, this precisely shows that $\phi(x)\in (L_S)_i^\chi$, as required.    
\end{proof}

\begin{Prp}\label{Pautgr} Let $L$ be a TKK-algebra of Brown type $\E_8$ over $R$, and denote its grading character by $\chi$. Then the inclusion $\bC_\bG(\chi)\to\bAutgr(L)$ is an isomorphism. Consequently, $\bAutgr(L)$ is smooth with connected fibres.
\end{Prp}

\begin{proof} Both $\bC_\bG(\chi)$ and $\bAutgr(L)$ are finitely presented, and $\bC_\bG(\chi)$ is smooth, hence flat. Therefore, \cite[$_4$17.9.5]{EGAIV} reduces the isomorphism statement to the case where $R$ is a field, where it is contained in \cite[4.2]{Rig}. Smoothness then holds by the above lemma. For any $R$-field $K$, connectedness of $\bAutgr(L)_K$ follows by \cite[4.2]{Rig} from the connectedness of $\bG$.
\end{proof}

\subsection{Isotopies and similitudes}
When $L=L(B)$ for a Brown algebra $B$, we will compare the graded automorphism group to the structure group $\bStr(B)$ consisting of the autotopies of $B$. We recall the definition.

\begin{Def} An \emph{isotopy} from a structurable algebra $B$ with $V$-operator $V$ to a structurable algebra $B'$ with $V$-operator $V'$ (both over $R$) is an $R$-linear bijection $\phi: B\to B'$ for which there exists an $R$-linear bijection $\widehat\phi:B\to B'$ such that
\[\phi V_{x,y}=V'_{\phi(x),\widehat\phi(y)}\phi\]
for all $x,y\in B$.
\end{Def}
From \cite[Section 8]{AH}, we know that $\widehat\phi$ is uniquely determined by $\phi.$

We then define $\bStr(B)$ by
\[\bStr(B)(S)=\{\phi\in \bGL_S(B_S)\mid \phi:B_S\to B_S \text{\ is an isotopy}\}\]
for each $R$-ring $S$. Isotopies from an algebra to itself are called \emph{autotopies}.

\begin{Ex} If $j\in B^-$ satisfies $j^2=\mu1_B\in R^*1_B$, then it follows from Lemma \ref{Lj} that the map $L_j:B\to B, x\mapsto jx$, is an autotopy, with $\widehat{L_j}=L_{-\mu^{-1}j}=L_{\widehat j}$.
    
\end{Ex}

The following useful lemma is a slight generalization of  a result from \cite[Proposition 12.3]{AH}. 

\begin{Lma}\label{Lstructurable} Let $B$ and $B'$ be structurable algebras that are finitely generated projective as $R$-modules. Let $\phi:L(B)\to L(B')$ be a graded isomorphism and denote its restriction to $L(B)_i$ by $\phi_i$. Then $\phi_1$ is an isotopy.
\end{Lma}

\begin{proof} We will show that $\phi_1V_{x,y}z=V'_{{\phi_1(x)},{\phi_{-1}(y)}}\phi_1(z)$ for all $x,y,z\in B$. We let $L=L(B)$ and focus on $L_{-1}\oplus L_0\oplus L_1$ and $L'_{-1}\oplus L'_0\oplus L'_1$ whose elements we write as triples. If now $x,z\in L_1=B$ and $y\in L_{-1}=B$, then by the definition of the Lie bracket on $L$,
\[[[(0,0,x),(y,0,0)],(0,0,z)]=(0,0, V_{x,y}z).\]
Since $\phi$ is a graded isomorphism, applying it yields
\[(0,0, V'_{\phi_1(x),\phi_{-1}(y)}\phi_1(z))=(0,0, \phi_1V_{x,y}z),\]
and the claim follows.
\end{proof}

In case $B=B'$, this implies that restriction to $L_1$ defines a map $\bAutgr(L(B))\to\bStr(B)$ by $\phi\mapsto\phi_1$. By construction, it is a natural transformation that respects composition, hence a homomorphism of affine group schemes. In fact, it is more.

\begin{Prp}\label{Pstrgr} The restriction map $\bAutgr(L(B))\to\bStr(B)$ is an isomorphism of affine group schemes.
\end{Prp}

\begin{proof} Both groups are finitely presented, and $\bAutgr(L(B))$ is smooth, hence flat, by Proposition \ref{Pautgr}. By the fibre-wise isomorphism criterion \cite[$_4$17.9.5]{EGAIV}, it is enough to show isomorphism over fields, which was done in \cite[Proposition 12.3]{AH}. (See also \cite[5.5]{Rig}.)
\end{proof}

\begin{Prp} Let $B$ be a Brown algebra over $R$. The group $\bStr(B)$ is a reductive group scheme.
\end{Prp}

\begin{proof} The above proposition, together with Proposition \ref{Pautgr}, implies that $\bStr(B)$ is a smooth affine group scheme whose geometric fibres (i.e.\ $\bStr(B)_k=\bStr(B_k)$ for any algebraically closed $R$-field $k$) are connected. Since the geometric fibres are isomorphic to $\bC_\bG(\chi)_k=\bC_{\bG_k}(\chi_k)$, and centralizers of tori are reductive over algebraically closed fields, the assertion follows.    
\end{proof}

Given a Freudenthal triple system $Q=(M,b,t)$, its invariance group scheme $\bInv(Q)$ is defined as that subgroup of $\bGL(M)$ that respects $b$ and $t$ in the sense that, for each $R$-ring $S$, a linear map $\phi:M\to M$ belongs to $\bInv(Q)$ if and only if
\[b(\phi(x),\phi(y))=b(x,y), \qquad \text{and}\qquad \phi t(x,y,z)=t(\phi(x),\phi(y),\phi(z))\]
for all $x,y,z\in M_S$. Recall that to a Brown algebra $B$ and a conjugate invertible element $j\in B^-$, we have associated in Remark \ref{RFTS} the Freudenthal triple system $Q(B,j)$. The following was proved in \cite{GPR}.

\begin{Lma}\label{Smooth} The group $\bInv(Q(B,j))$ is a simple, simply connected affine group scheme of type $\E_7$.
\end{Lma}

The group $\bStr(B)$ is related to similitudes of $Q=Q(B,j)$ and to $\bInv(Q(B,j))$. To this end, consider the similitude group $\bSim(Q)$, which is the group $R$-functor such that, for each $S\in \Ralg$, $\bSim(Q)(S)$ consists of all $\phi\in\bGL(B)(S)$ that satisfy
\[b_j(\phi(x),\phi(y))=m_\phi b_j(x,y)\quad\text{and}\quad t_j(\phi(x),\phi(y),\phi(z))=m_\phi\phi(t_j(x,y,z))\]
for some $m_\phi\in S^*$. We then say that $\phi$ is a \emph{similitude} of $Q$ with \emph{multiplier} $m_\phi$.

\begin{Prp}\label{PInv} Let  $j\in B^-$ be conjugate invertible with $\mu=j^2\in R^*1_B$, and set $Q=Q(B,j)$.
\begin{enumerate}[(i)]
\item For each $R$-ring $S$, every $\phi\in\bStr(B)(S)$ is a similitude of $Q_S$  with multiplier $m_\phi=\tfrac12\mu^{-1}b_j(\phi(j),\phi(1))$.
\item The groups $\bSim(Q)$ and $\bStr(B)$ coincide.
    \item The group $\bInv(Q)$ is the closed subgroup of $\bStr(B)$ satisfying
    \[\bInv(Q)(S)=\{\phi\in\bStr(B)(S)\mid \forall x,y\in B_S: b_j(\phi(x),\phi(y))=b_j(x,y)\}\]
for each $R$-ring $S$. Moreover, the condition on $\phi\in\bStr(B)(S)$ that
\[\forall x,y\in B_S: b_j(\phi(x),\phi(y))=b_j(x,y)\]
is equivalent to $b_j(\phi(j),\phi(1))=b_j(j,1)$ and to $b_j(\phi(j),\phi(1))=2\mu$.
\item Each $\phi\in \bInv(Q)(S)$ satisfies $b_j(\widehat\phi(x),\widehat\phi(y))=b_j(x,y)$.
\end{enumerate}
\end{Prp}

\begin{proof}

To prove (i), we use two identities from \cite[(61)-(62)]{AH}, namely
\[(xy^*-yx^*)w= V_{x,w}y-V_{y,w}x\]
and
\[\phi(jw)=\tfrac12(\phi(j)\phi(1)^*-\phi(1)\phi(j)^*)\widehat\phi(w)
\]
for all $x,y,w\in B_S$. From the definition of $b_j$, it moreover follows that 
\[xy^*-yx^*=\mu^{-1}b(x,y)j.\]
Let $\phi\in\bStr(B)(S)$ and let $w=\widehat\phi^{-1}(j)$. Using the above identities gives
\[\begin{array}{rcl}
b_j(\phi(x),\phi(y))&=&(\phi(x)\phi(y)^*-\phi(y)\phi(x)^*)j\\
&=& V_{\phi(x),j}\phi(y)-V_{\phi(y),j}\phi(x)\\
&=& \phi(V_{x,w}y-V_{y,w}x)\\
&=& \phi((xy^*-yx^*)w)\\
&=& \phi(\mu^{-1}b(x,y)jw)\\
&=& \tfrac12\mu^{-1}b(x,y)(\phi(j)\phi(1)^*-\phi(1)\phi(j)^*)\widehat\phi(w)\\
&=& \tfrac12\mu^{-1}b(x,y)(\phi(j)\phi(1)^*-\phi(1)\phi(j)^*)j\\
&=& \tfrac12\mu^{-1}b(x,y)b_j(\phi(j),\phi(1)).\\
\end{array}\]
This proves the $b_j$-part of being a similitude with $m_\phi=\tfrac12\mu^{-1}b_j(\phi(j),\phi(1))$. From this and the definition of $t_j$, it follows that $\phi$ is a similitude provided that
\[m_\phi\phi V_{x,jy}z=V_{\phi(x),j\phi(y)}\phi(z).\]
Let us show this. Since $\phi$ is an autotopy, we know that 
\[\phi V_{x,jy}z=V_{\phi(x),\widehat\phi(jy)}\phi(z).\]
Thus we need to show that $j\phi(y)=m_\phi\widehat\phi(jy)$ for all $y$. Replacing $y$ by $jy$, multiplying by $j$ from the left, and using $j^2=\mu$, we see that this is equivalent to $\phi(jy)=m_\phi j\widehat\phi(y)$. By the above identities from \cite{AH},
\[\phi(jy)=\tfrac12(\phi(j)\phi(1)^*-\phi(1)\phi(j)^*)\widehat\phi(y).\]
Now,
\[  (\phi(j)\phi(1)^*-\phi(1)\phi(j)^*)j= b_j(\phi(j),\phi(1))=2\mu m_\phi.\]
Right multiplication by $\tfrac12\mu^{-1}j$ thus implies $\tfrac12(\phi(j)\phi(1)^*-\phi(1))=m_\phi j$, so $\phi(jy)=m_\phi j\widehat\phi(y)$ as desired. This proves (i).

To prove (ii), it thus remains to prove the inclusion $\bSim(Q)(S)\subseteq \bStr(B)(S)$. If $\phi$ is a similitude of $Q_S$ with multiplier $m_\phi$, then using the definition of $t_j$,
\[\begin{array}{rcl}
    2m_\phi\phi(V_{x,jy}z)&=&m_\phi\phi(t_j(x,y,z))+m_\phi b_j(y,z)\phi(x)\\
    &+&m_\phi b_j(y,x)\phi(z)+m_\phi b_j(x,z)\phi(y)\\
    &=& t_j(\phi(x),\phi(y),\phi(z))+b_j(\phi(y),\phi(z))\phi(x)\\
    &+& b_j(\phi(y),\phi(x))\phi(z)+b_j(\phi(x),\phi(z))\phi(y)\\
    &=& 2V_{\phi(x),j\phi(y)}\phi(z).
\end{array}\]
Defining $\psi$ by $\psi(w)=(m_\phi\mu)^{-1}j\phi(jw)$ we get $j\phi(y)=m_\phi\psi(jy)$. The above then implies that $\phi$ is an autotopy with $\widehat\phi=\psi$. This proves (ii).

Combining (i) and (ii), an element $\phi\in\bStr(B)(S)=\bSim(Q)(S)$ is in $\bInv(Q)(S)$ if and only if $m_\phi=1$, which happens if and only if $\phi$ is an isometry of $b_j$. This proves (iii), where the last part follows from $b_j(j,1)=2\mu$.

To prove (iv), we use, on the one hand, the equality $j\widehat\phi(y)=\phi(jy)$ from above, which is equivalent to $\widehat\phi(y)=\mu^{-1}j\phi(jy)$, and on the other hand, the equality $b_j(jx,jy)=-\mu b_j(x,y)$, which follows from \cite[Lemma 11.2]{AH}. We get
\[\begin{array}{rcl}
b_j(\widehat\phi(x),\widehat\phi(y))&=&\mu^{-2}b_j(j\phi(jx),j\phi(jy))\\
&=& -\mu^{-1}b_j(\phi(jx),\phi(jy))\\
&=& -\mu^{-1}b_j(jx,jy)\\
&=& b_j(x,y),
\end{array}\]
as desired.
\end{proof}

Note that for any $R$-ring $S$ and any $\phi,\psi\in\bStr(B)(S)$ we have
\[m_{\psi\phi}b(x,y)=b_j(\psi\phi(x),\psi\phi(y))=m_{\psi}b(\phi(x),(y))=m_{\psi}m_{\phi}b(x,y)\]
for all $x,y\in B_S$. Thus we have a group homomorphism
\[m:\bStr(B)\to\Gm, \quad \phi\mapsto m_{\phi}\]
and $\bInv(Q)=\ker\ m$. If $B$ is reduced, this map admits a section, as the following proposition shows. The notation is as in Example \ref{Ereduced}

\begin{Prp} Let $B$ be a reduced Brown algebra, set $j=\diag(1,-1)\in B^-$ and let $Q=Q(B,j)$. The sequence
\[\bone\to\bInv(Q)\to\bStr(B)\overset{m}\rightarrow\Gm\to\bone\]
is a split exact sequence of affine group schemes.    
\end{Prp}

\begin{proof} We have already established exactness at $\bInv(Q)$ and $\bStr(B)$. The remainder of the statement will follow by constructing a section $\Gm\to\bStr(B)$. Following \cite[Example 16.9]{GPR}, we define a group homomorphism $\Gm\to\bStr(B)$ as follows: for each $S\in\Ralg$, it maps $\lambda\in S^*$ to
\[\phi_\lambda:B\to B,\quad \begin{pmatrix}
 r_1&a_1\\ a_2& r_2   
\end{pmatrix}\mapsto \begin{pmatrix}
 \lambda^{-1}r_1&\lambda a_1\\ a_2& \lambda^2r_2   
\end{pmatrix}.\]
From  \cite[Example 16.9]{GPR} we know that $b_j(\phi_\lambda(x),\phi_\lambda(y))=\lambda b_j(x,y)$ and \linebreak $t_j(\phi_\lambda(x),\phi_\lambda(y),\phi_\lambda(z))=\lambda\phi_\lambda(t_j(x,y,z))$. From the definition of $t_j$ it follows that
\[\phi_\lambda(V_{x,y}z)=V_{\phi_\lambda(x),\psi(y)}\phi_\lambda z)\quad \text{where}\quad \psi(y)=\lambda^{-1}j\phi_\lambda(jy).\]
Thus $\phi_\lambda\in\bStr(B)(S)$ and $m_{\phi_\lambda}=\lambda$ as desired.
\end{proof}

This has the following consequence.

\begin{Cor}\label{Ctrivker} Given a reduced Brown algebra $B$, $j=\diag(1,-1)\in B^-$ and $Q=Q(B,j)$, let $f:H^1_{\mathrm{fppf}}(R,\bInv(Q))\to H^1_{\mathrm{fppf}}(R,\bStr(B))$ be induced by the inclusion $\bInv(Q)\to\bStr(B)$. Then the kernel of $f$ is trivial.
\end{Cor}

\begin{proof} By the previous proposition, $\bStr(B)$ acts transitively on $\Gm$ by $\phi\cdot \lambda=m_\phi\lambda$. The stabilizer of $1\in\Gm(R)$ is $\bInv(Q)$. Since these group schemes are all smooth of finite presentation, it follows that $\bStr(B)/\bInv(Q)\simeq\Gm$.  Thus by \cite[Proposition 2.4.3]{Gil2}, $\ker f$ is in bijection with the set of $\bStr(B)(R)$-orbits of $\Gm(R)$, which is a singleton set by transitivity. This completes the proof.
\end{proof} 

Summing up the earlier results, we have the following chain of inclusions of closed subgroups whenever $j\in B^-$ is conjugate invertible:
\[\bInv(Q(B,j))\to\bSim(Q(B,j))=\bStr(B)\overset\sim\to \bAutgr(L(B))\to\bAut(L(B)).\]

We will mainly be interested in the homogeneous space $\bAut(L(B)/\bInv(Q(B,j))$, viewed as an fppf-quotient, and the twists by the torsors that it defines. To parametrize this quotient, we will construct a scheme using extremal elements. 

\section{Extremal elements}
We start by establishing the basics of extremal elements in Lie algebras over rings. We shall require these elements to be unimodular; recall that an element $x$ in an $R$-module $M$ is \emph{unimodular} if there is a linear form $\alpha:M\to R$ such that $\alpha(m)=1$. This requirement avoids degenerate cases and is, in this setting, the natural generalization to rings of the property of being non-zero over fields.

\begin{Def}\label{Dextremal} Let $L$ be a Lie algebra over $R$.
\begin{enumerate}[(i)]
    \item An $R$-submodule $I\subseteq L$ is said to be an \emph{inner ideal} of $L$ if $[I,[I,L]]\subseteq I$.
    \item An element $x\in L$ is said to be \emph{extremal} if $x$ is unimodular and $Rx$ is an inner ideal of $L$.
\end{enumerate}
\end{Def}

Thus a unimodular element $x\in L$ is extremal if and only if for each $y\in L$ there exists $r=r_x(y) \in R$ such that $[x,[x,y]]=r_x(y)x$. This $r$ is unique, for if $rx=r'x$, then $r-r'=0$ since it annihilates the unimodular element $x$. For each extremal element $x\in L$ we thus have a linear map $r_x:L\to R$.  Clearly $r_x(x)=0$. 

\begin{Ex} If $L=\sll_2(R)$, with its usual Chevalley basis $(e,f,h)$, then $e$ and $f$ are extremal, and we have
\[r_e(f)=r_f(e)=-2 \qquad\text{and}\qquad r_e(h)=r_f(h)=0.\]
\end{Ex}

\begin{Rk}\label{Rexiso} If $x\in L$ is extremal and $\phi:L\to L'$ is an isomorphism of Lie algebras, then $\phi(x)$ is unimodular since $x$ is, and satisfies
\[[\phi(x),[\phi(x),y]]=\phi([x,[x,\phi^{-1}(y)]])=\phi(r_x(\phi^{-1}(y))x)=r_x(\phi^{-1}(y))\phi(x)\]
for all $y\in L'$. It follows that $\phi(x)$ is extremal with $r_{\phi(x)}=r_x\phi^{-1}$.
\end{Rk}

Extremal elements in a Lie algebra $L$ are stable under base change and define a set-valued functor as follows.

\begin{Lma} Let $L$ be an $R$-Lie algebra whose underlying module is finitely generated. For each $S\in\Ralg$, set
\[\bX^L(S)=\{x\in L_S \mid x \text{\ is an extremal element of\ } L_S\}.\]
Then $\bX^L$ is a full subfunctor of the set-valued $R$-functor $\bL$ that is defined by $S\mapsto L_S$ on objects and $\rho\mapsto \Id_L\otimes \rho$ on morphisms $\rho:S\to T$.
\end{Lma}

\begin{proof} It suffices to show that if $S\in\Ralg$, $x\in L_S$ is extremal, and $T$ is an $S$-ring, then $x\otimes 1_T\in L_T$ is extremal. Let $e_1,\ldots,e_n$ be a set of generators of the $R$-module $L$. If $y\in L_T$, then $y=\Sigma_i e_i\otimes t_i\in L_T$, and so
\[[x\otimes 1,[x\otimes 1,y]]=\Sigma_i[x,[x, e_i]]\otimes t_i=\Sigma_ir_x(e_i)x\otimes t_i=\Sigma_ix\otimes r_x(e_i)t_i=t(x\otimes 1)\]
where $t=\Sigma_ir_x(e_i)t_i$. 
\end{proof}

\subsection{Extremal elements in TKK-algebras}
We now specialize to the case where $L$ is a TKK-algebra.

\begin{Ex}\label{Eextremal} If $L=L(B)$ is a TKK-algebra over $R$ containing a conjugate invertible element $j\in B^-$, we denote $j$ by $1_+$ in $L_2$ and by $1_-$ in $L_{-2}$, and similarly in any base change $L_S$ of $L$. In fact, $1_\pm\in\bX^L(S)$, by virtue of the $\Z$-grading: indeed, the only non-trivial case to check, when $x=1\pm$, is when $y=s_{\mp2}=s1_{\mp2}\in L_{\mp2}$. Then 
\[[x,[x,y]]=[1_\pm,[1_\pm,s_{\mp2}]]= \pm s[1_{\pm}, V_{\mu,1}]= -2\mu s_{\pm}.\]
\end{Ex}

 If $x$ is extremal and $y\in L$, then the scalar $r_x(y)$ is of interest. Over fields, we have the following from \cite[Lemma 2.4, Theorem 2.5, Proposition 3.3 and Proposition 9.14]{CSUW}.

\begin{Prp} Let $K$ be a field (of characteristic not 2 or 3) and $L$ a simple Lie algebra of Chevalley type over $K$. Then $L$ is spanned, as a $K$-vector space, by its extremal elements. Moreover, there is a unique symmetric bilinear form $r^K:L\times L\to K$ that satisfies $r^K(x,y)=r_x(y)$ whenever $x$ is extremal. The form $r^K$ is associative in the sense that $r^K([x,y],z)=r^K(x,[y,z])$, and regular.
\end{Prp}

This allows us to express $r_x(y)$ in terms of $\tau(x,y)$ over fields and, more generally, over integral domains, as follows.

\begin{Prp}\label{Ptaur} Let $R$ be an integral domain, $L$ a TKK-algebra over $R$, $x\in \bX^L(R)$ and $y\in L$. Then $[x,[x,y]]=-2\tau(x,y)x$, i.e.\ $r_x(y)=-2\tau(x,y)$
\end{Prp}
\begin{proof}
    As $R$ is an integral domain, there is an embedding $R\hookrightarrow k$ with $k$ an algebraically closed field. In that case, $L_k$ is the Chevalley $k$-Lie algebra of type $E_8$. Thus the above proposition applies, and we have two regular, associative symmetric bilinear forms on $L_k$, namely $r^k$ and $\tau_k=\tau\otimes_R1_k$. Thus $r^k=\lambda\tau_k$ for some $\lambda\in k^*$. To determine $\lambda$, we set $x=1_+$ and $y=1_-$. Then from Example \ref{Eextremal} we have $[x,[x,y]]=-2\mu y$, so that $r^k(x,y)=-2\mu$, while $\tau_k(x,y)=j^2=\mu$. Thus $\lambda=-2$, and $r^k=-2\tau_k$ on $L_k$. Now let $x \in \mathbf{X}^L(R)$. Since $x$ is unimodular, $x\otimes 1$ is a non-zero extremal element in $L_k=L\otimes_R k$, so for all $y\in L$,
    \[r_x(y)\otimes 1=r_{x\otimes 1}(y\otimes1)=r^k(x\otimes 1,y\otimes1)=-2\tau_k(x\otimes 1,y\otimes 1)=-2\tau(x,y)\otimes 1,\]
    whence we have $r_x(y) = -2\tau(x,y)\in R$ for all $y\in L$ as desired.
\end{proof}

When $R$ is not an integral domain, the scalar multiple $-2$ is shifted by a nilpotent element, as the following lemma shows. Later on, we will be interested in certain pairs of extremal elements, and this lemma will be used to prove that for those elements, $r_x=-2\tau(x,-)$.

\begin{Lma}\label{Llambda} Let $x\in \bX^L(R)$. Then there exists $c\in L$ such that $r_x(c)\in R^*$ and for all $y\in L$, $r_x(y)=r_x(c)\tau(x,y)$. Thus there is a unique $\lambda_x\in R^*$ such that, for all $y\in L$, $[x,[x,y]]=\lambda_x\tau(x,y)x$.
\end{Lma}

\begin{proof} By unimodularity of $x$ and regularity of $\tau$ there exists $c\in L$ such that $\tau(x,c)=1$. For any $y\in L$, we have, on the one hand
\[\tau([x,[x,y]],c)=\tau(r_x(y)x,c)=r_x(y)\tau(x,c)=r_x(y).\]
On the other hand, since $\tau$ is symmetric and associative, we have   
\[\tau([x,[x,y]],c)=\tau([c,x],[x,y])=\tau([[c,x],x],y)=\tau(r_x(c)x,y)=r_x(c)\tau(x,y).\]
Altogether, $r_x(y)=r_x(c)\tau(x,y)$ for all $y\in L$. Let $n=r_x(c)+2\tau(x,c)=r_x(c)+2$. Then by the previous proposition, $n(P)=0$ for every prime ideal $P$ in $R$, so $n$ is nilpotent. Since $2\in R^*$, this implies that $r_x(c)=n-2$ is invertible. This proves the first claim, and the existence part of the second claim. For uniqueness, assume that $\lambda_x\tau(x,y)x=\lambda'\tau(x,y)x$ for all $y$. Then with $y=c$ as above we get $\lambda_xx=\lambda'x$, which proves uniqueness since $x$ is unimodular.
\end{proof}

Notice how for integral domains, we did not need the associativity of $\tau$, and the fact that $\tau$ is associative then follows from the associativity of $r$. In the general case, however, we did need $\tau$ to be associative \emph{a priori}.

\subsection{Schemes of extremal elements}
Our goal is to define an affine scheme of certain pairs of extremal elements in $L$, starting from $\bX^L$ above. Towards this goal, we need a few intermediate steps. 

First, we define, for any $R$-ring $S$,
\[\widetilde\bX(S)=\{(x,\lambda)\in L_S\times S^* \mid F_{x,\lambda}:L_S\to L_S \text{\ is the zero map}\} \]
where $F_{x,\lambda}$ is the linear map defined by
\[F_{x,\lambda}(y)=[x,[x,y]]-\lambda\tau(x,y)x. \]
Note that if $x\in L_S$ is unimodular, then $x\in \widetilde\bX(S)$ if and only if $x$ is extremal with $r_x(y)=-\lambda\tau(x,y)$. By Lemma \ref{Llambda}, this holds for a unique value of $\lambda$.
\begin{Rk} Over fields where $30$ is invertible, the reader might wish to compare this to the Freudenthal map $x\times x:y\mapsto [x,[x,y]]+\tfrac{1}{30}\kappa(x,y)x$ (see \cite{YIS}), which agrees with $F_{x,-2}$ in view of Corollary \ref{Ckillingform}. More generally, if $S$ is an integral domain (where $5$ need not be invertible), $(x,\lambda)\in \widetilde\bX(S)$ only if $\lambda=-2$, in view of Proposition \ref{Ptaur}.
\end{Rk}

\begin{Lma} The assignment $S\mapsto \widetilde\bX(S)$ defines a full subfunctor of $\bL\times \Gm$. It is represented by an affine scheme over $R$.
\end{Lma}

\begin{proof} Recall that $L$ is a finitely generated $R$-module and fix a set of generators $e_1,\ldots e_n$. Then the $(e_i)_S=e_i\otimes_RS$ generate $L_S$ for any $S\in\Ralg$.

For functoriality, we need to prove that if $(x,\lambda)\in \widetilde\bX(S)$ and $T$ is an $S$-ring, then ($x_T,\lambda_T)\in\widetilde\bX(T)$, i.e.\ $F_{x_T,\lambda_T}(y)=0$ for all $y\in L_T$.  If $y\in L_T$, then writing $y=\sum_i (e_i)_S\otimes_S t_i$ we have
    \[F_{x_T,\lambda_T}(y)=\sum_i t_i\bigg([x\otimes_S 1,[x\otimes_S1,e_i\otimes_S 1]]-(\lambda\otimes_S1)\tau_T(x\otimes_S1, e_i\otimes_S 1)(x\otimes_S 1)\bigg),\]
    which is equal to $\sum_i t_i(F_{x,\lambda}((e_i)_S)\otimes_S 1)$ which is zero since $F_{x,\lambda}$ is the zero map.

    To show that it is an affine scheme, assume first that $L$ is a free module and that $e_1,\ldots, e_n$ is an $R$-basis. By linearity, $(x,\lambda)\in \widetilde\bX(S)$ if and only if $F_{x,\lambda}((e_i)_S)=0$ for all $i$. This is a system of polynomial equations, with coefficients in $R$, in $\lambda$ and the coordinates of $x$. Thus $\widetilde\bX$ is a closed subscheme of the affine scheme $\bL\times \Gm$, hence affine.

    In the general case, there exists a faithfully flat $R$-algebra $S$ such that $L_S$ is free, so that $\widetilde\bX_S$ is an affine scheme over $S$. Then $\widetilde\bX$ is an affine scheme over $R$ by descent.
\end{proof}

We will now construct what will be the base scheme for our torsor. We thus define $\bY=\bY^L$ and $\widetilde\bY$ by
\[\bY^L(S) \coloneqq \{(x,y) \in \bX^L(S)\times \bX^L(S) \mid r_x(y)=-2\}\]
and
\[\tY(S) \coloneqq \{((x,\lambda),(y,\nu)) \in \tX(S)\times \tX(S) \mid r_x(y)=-2\}.\]

Since $\bX^L$ and $\tX$ are affine schemes, so are $\bY^L$ and $\tY$, and if $((x,\lambda),(y,\nu))\in\tY(S)$, then Lemma \ref{Llambda} implies that $\tau(x,y)\in S^*$, and hence that $x$ and $y$ are unimodular. 

\begin{Ex} If $L=L(B)$ and $j\in B^-$ satisfies $j^2=\mu1_B$ with $\mu\in R^*$, then by Example \ref{Eextremal}, we have $(1_+,\mu^{-1}1_-)\in\bY^L(S)$.    
\end{Ex}

To show that these schemes are isomorphic, consider the natural map $p:\bL\times\Gm\to\bL$ defined, for each $R$-ring $S$, by $(x,\lambda)\mapsto x$.

\begin{Lma} The map $p\times p$ induces an isomorphism $\tY\to\bY^L$.
\end{Lma}

\begin{proof} We need to show that $p\times p$ induces a bijection $\tY(S)\to\bY^L(S)$ for each $S\in\Ralg$. If $((x,\lambda),(y,\nu))\in\tY(S)$, then $x$ and $y$ are extremal,
so 
\[(p\times p)_S((x,\lambda),(y,\nu))\in\bY^L(S).\]
Conversely, if $(x,y)\in \bY^L(S)$, then $x$ and $y$ are extremal, so by Lemma \ref{Llambda}, there is a unique pair $(\lambda,\nu)\in S^*\times S^*$ such that $(x,\lambda), (y,\nu)\in\widetilde\bX(S)$. Thus there is a unique $((x,\lambda),(y,\nu))\in \tY(S)$ with $(p\times p)_S((x,\lambda),(y,\nu))=(x,y)$. This completes the proof.
\end{proof}

\section{The orbits of the $\bG$-action}
In this section, we will study the behaviour, under the action of $\bG=\bAut(L)$, of the schemes defined in the previous section. In the end, we will have shown that $\bY^L$ is smooth and parametrizes the quotient $\bG/\bInv(Q(B,j))$ when $L=L(B)$ and $j\in B^-$ is conjugate invertible.

\begin{Lma} The natural action of $\bG=\bAut(L)$ on $\bL$ induces an action on $\bX^L$ and an action on $\bY^L$.
\end{Lma}

\begin{proof} Let $\phi\in\bG(S)$ and $x\in\bX^L(S)$. Then $\phi(x)\in\bX^L(S)$ by Remark \ref{Rexiso}. Next let $(x,y)\in\bY^L(S)$. Then $(\phi(x),\phi(y))\in \bX^L(S)\times\bX^L(S)$ and
\[[\phi(x),[\phi(x),\phi(y)]]=\phi([x,[x,y]])=\phi(-2x)=-2\phi(x),\]
i.e.\ $r_{\phi(x)}(\phi(y))=-2$. Thus  $(\phi(x),\phi(y))\in \bY^L(S)$.
\end{proof}

The next result depicts the situation over fields of characteristic not 2 or 3. 

\begin{Prp}\label{Trans} Assume that $k\in\Ralg$ is a field such that $L_k=L(B)$ for a Brown algebra $B$ over $k$, and let $\mu=j^2$ where $j$ spans $B^-$. Then $\bG(k)$ maps any $(a,b)\in\bY^L(k)$ to $(\lambda 1_+,(\lambda\mu)^{-1}1_-)$ for some $\lambda\in k^*$. Moreover, if $k$ is algebraically closed, then $\bG(k)$ acts transitively on $\bY^L(k)$.
\end{Prp}

Note that when $k$ is algebraically closed, we have  $L_k\simeq L(B_k^s)$, where $B_k^s$ is the split Brown algebra over $k$.

\begin{proof} Let $(x,y)\in \bY^L(k)$. Since $x\in\bX^L(k)$, the $k$-span of $x$ is an inner ideal of $L_k$, hence necessarily a minimal inner ideal, which is non-zero since $x$ is unimodular. By \cite[Lemma 3.7]{DMM}, there is an automorphism $\phi\in\bG(k)$ such that the $+2$-degree component of $\phi(x)$ is non-zero. Then \cite[Lemma 3.9]{DMM} implies that there is $\psi\in\bG(k)$ mapping $k\phi(x)$ to $L_2$, and thus mapping $\phi(x)$ to $\lambda 1_+$ for some $\lambda\in k^*$. 

Altogether, this shows that $(x,y)$ is in the same orbit as $(\lambda1_+,z)\in\bY^L(k)$ for some $z\in \bX^L(k)$ with $[\lambda1_+,[\lambda1_+,z]]=-2\lambda1_+$. This latter condition in particular implies that the $-2$-graded component of $z$ is non-zero, so by \cite[Lemma 3.9]{DMM} (applied after switching signs), there is an automorphism fixing $1_+$ and mapping $kz$ to $L_{-2}$, thus mapping $z$ to $\lambda'1_-$ for some $\lambda'\in k^*$. Together with the above, this shows that $(x,y)$ is in the same orbit as $(\lambda1_+,\lambda'1_-)$, and this element satisfies $[\lambda1_+,[\lambda1_+,\lambda'1_-]]=-2\lambda1_+$. From Example \ref{Eextremal} we know that \[[\lambda1_+,[\lambda1_+,\lambda'1_-]]=-2\lambda^2\lambda'\mu1_+,\]
so altogether $\lambda'=(\lambda\mu)^{-1}$.  This proves the first statement.

If $k$ is algebraically closed, $B_k$ is split, so we may take $j$ with $\mu=j^2=1$. Then any $(x,y)\in \bY^L(k)$ is in the orbit of $(\lambda1_+,\lambda^{-1}1_-)$ for some $\lambda\in k^*$, and $\lambda^{-1}=\nu^2$  for some $\nu\in k^*$. The grading cocharacter $\chi$ then provides us with an automorphism $\chi(\nu)$ of $L_k$ mapping $(\lambda1_+,\lambda^{-1}1_-)$ to $(1_+,1_-)$. This completes the proof.
\end{proof}

Next we will prove that over any algebraically closed field $k$, $\bY^L$ is smooth. First, we consider the principal open subschemes $\bX^+$ and $\bX^-$ of $\bX^L$  (over an arbitrary ring $R$) corresponding to the linear (hence polynomial) functions $r_{1_+}$ and $r_{1-}$, respectively. In other words, for all $R$-rings $S$,
\[\bX^\pm(S)=\{x\in\bX^L(S) \mid r_{1_\pm}(x)\in S^*\}.\]

Since $k$ is algebraically closed, the subspace $B^-$ is spanned by $j$ with $j^2=1$. We shall use this in Proposition \ref{Psmooth} and Corollary \ref{Csmooth}.

\begin{Prp}\label{Psmooth} Over any algebraically closed field $k$ of characteristic not 2 or 3, the schemes $\bX^\pm$ are affine open subschemes of $\bX^L$ that are isomorphic to $\mathbb A^{57}\times \Gm$. In particular, they are smooth.
\end{Prp}

Note that the affine space $\mathbb A^{57}$ in question is the product of the underlying space of the Brown algebra $B$, viewed in $L=L(B^s)$ as $L_{-1}$ or $L_{1}$ as the case may be, and one copy of the affine line, viewed as the span of $V_{1,1}\in L_0$. For the proof, we note that if $x\in L_1$, then $\ad_x^n=0$ for $n>4$, so that
\[\mathrm{Ad}_x=\textstyle\sum_{n=0}^4 \frac{1}{n!}\ad_x^n\]
is well defined as all denominators are invertible. By \cite[Theorem 2.10]{Sta}, it is an automorphism of $L$. Its inverse is $\mathrm{Ad}_{-x}$.

 \begin{proof} We argue for $\bX^+$; the argument for $\bX^-$ is identical, \emph{mutatis mutandis}. The function $r_{1_+}:L\to k$ is polynomial and induces a regular function $f$ on the affine scheme $\bX^L$. Then $\bX^+$ is the distinguished affine open subscheme $\bX_f$ of $\bX^L$.

 To establish the desired isomorphism, consider, for each $k$-ring $S$, the map $\eta_S:\bX^+(S)\to B_S\times S\times S^*$
 defined by
 \[x\mapsto (x_{-1},\tfrac12\tau(\mathrm{Ad}_{(0,0,0,js^{-1}x_{-1},0)}x,V_{1,1}),s) \qquad{where}\qquad x_{-2}=s1_-.\]
This is well-defined since $r_{1_+}(x)=-2s$, so that $s\in S^*$. The coordinates are regular functions, and this defines a natural transformation $\eta:\bX^+\to\mathbb A^{57}\times\Gm$. In the other direction we define the map $\zeta_S: B_S\times S\times S^*\to \bX^+(S)$ by
\[(b,r,s)\mapsto \mathrm{Ad}_{-(0,0,0,js^{-1}b,0)}(s1_-,0,rV_{1,1},0,- r^2s^{-1}1_+).\]
This map is well-defined since $s\in S^*$ and a direct computation shows that $y=(s1_-,0,rV_{1,1},0,- r^2s^{-1}1_+)$ is extremal, and thus so is its automorphic image $x=\mathrm{Ad}_{-(0,0,0,js^{-1}b,0)}y$. Its coordinates are regular functions and it defines a natural transformation $\zeta:\mathbb A^{57}\times\Gm\to\bX^+$. We will show that $\zeta$ is inverse to $\eta$.

If $x=\zeta_S((b,r,s))$, then a direct check shows that $x_{-1}=b$ and $x_{-2}=s1_-$, so
\[\eta_S(x)=(b,\tfrac12\tau(y,V_{1,1}),s)=(b,\tfrac12r\beta(1,1),s)=(b,r,s).\]
Thus $\eta\zeta=\Id$. In the other direction, if $(b,r,s)=\eta_S(x)$, then by definition of $\eta_S$ we have $x_{-2}=s1_-$ and $x_{-1}=b$. Set
\[y=\mathrm{Ad}_{(0,0,0,js^{-1}b,0)}x \quad \Longleftrightarrow \quad x=\mathrm{Ad}_{-(0,0,0,js^{-1}b,0)}y.\]
Then $y_{-2}=s1_-$ and $y_{-1}=0$. Now $y$, being the automorphic image of $x$, is extremal, so 
\[[y,[y,1_+]]=-2r_y(1_+)y.\]
Writing $y_0=\Sigma_iV_{c_i,d_i}$, the 0-component of the left-hand side is $-2sqV_{1,1}$, where $q=\tfrac12\Sigma_i\beta(c_i,d_i)$, while the 0-component of the right-hand side is $-2sy_0$. Since $s\in S^*$ this implies that $y_0=qV_{1,1}$. Using this, the 1- and 2-components of the equality compute to
\[sy_1=-2sy_1\qquad \text{and}\qquad 4q^21_++2sy_2=-2sy_2.\]
From this follows that $y_1=0$ and $y_2=-s^{-1}q^21_+$. The condition $\eta_S(x)=(b,r,s)$ further implies that $q=r$, so that altogether
\[y=(s1_-, 0, rV_{1,1}, 0, -r^2s^{-1}1_+),\]
giving $\zeta_S((b,r,s))=x$. This shows that $\zeta\eta=\Id$, and the proof is complete. 
\end{proof}

\begin{Cor}\label{Csmooth} Over any algebraically closed field $k$, the scheme $\bY^L$ is smooth.
\end{Cor}

\begin{proof} We will show that the point $(1_+,1_-)\in\bY^L(k)$ is regular. Then Proposition \ref{Trans} implies that all $k$-rational points are regular. Since $k$ is algebraically closed, this implies that all closed points are regular and that the scheme is smooth. 

By the previous proposition, the product $\bX^L\times\bX^L$ contains the smooth open subscheme $\bX^-\times \bX^+$. Consider the affine subscheme $\bU$ of $\bX^-\times \bX^+$ defined by the equation $r_x(y)=-2$. It is an open subscheme of $\bY^L$, namely the intersection of the two distinguished open subschemes $\bY^L_g\cap\bY^L_h$, where
\[g(x,y)=r_{1_-}(x) \qquad \text{and}\qquad h(x,y)=r_{1_+}(y).\]
The point $p=(1_+,1_-)$ is contained in $\bU(k)$. Since $\bX^-\times\bX^+$ is smooth at $p$, and $\bU$ is the closed subscheme of $\bX^-\times\bX^+$ defined by $r_x(y)+2=0$, to show that $p$ is regular in $\bU$ it suffices to check the Jacobian criterion applied to $h(x,y)=r_x(y)+2=-2(\tau(x,y)-1)$ at $p$. Using the notation from Definition \ref{Dtau}, the derivative of $h(x,y)$ with respect to $s_+$ is $-2t_-$, which is equal to $-2\neq 0$ at $p$. Thus $p$ is regular in $\bU$ and, since $\bU$ is open in $\bY^L$, it is a regular point of $\bY^L$. This completes the proof.
\end{proof}

We now go back to working over a general ring $R$ with $6\in R^*$. In the remainder of this section, we assume that $L=L(B)$ where $B^-$ contains a conjugate invertible element $j$, so that $j^2=\mu\in R^*1_B$. Recall that this is equivalent to $B^-$ being free. As before, we write $1_+$ for $j\in L_2$, and $1_-$ for $j\in L_{-2}$. We also set $Q=Q(B,j)$.

\begin{Prp}\label{Pstab} Under the action of $\bG$ on $\bY^L$, the stabilizer of $(1_+,\mu^{-1}1_-)$ is $\bInv(Q)$.
\end{Prp}

\begin{proof}
   We first prove that the inclusion $\bInv(Q)\to\bG$ factors through the stabilizer. Denote the isomorphism $\bStr(B)\to\bAutgr(L)$ by $\iota$. Take $S\in \Ralg$ and let $\phi\in\bInv(Q)(S)\subseteq\bStr(B)(S)$. Then
   \[\phi V_{x,y}=V_{\phi(x), \widehat\phi(y)}\phi,\]
   and by Proposition \ref{PInv}, both $\phi$ and $\widehat\phi$ respect $b_j$.
   
   By \cite[Proposition 12.3]{AH}, $\iota(\phi)$, acting on $L_2=Rj$, satisfies
   \[2\iota(\phi)(j)=\phi(j)\phi(1)^* - \phi(1)\phi(j)^*=\mu^{-1}b_j(\phi(j),\phi(1))j=\mu^{-1}b_j(j,1)j=2j,\]
implying that $\iota(\phi)$ stabilizes $1_+$. Similarly, $\iota(\phi)$, acting on $L_{-2}$, satisfies
   \[2\iota(\phi)(j)=\widehat\phi(j)\widehat\phi(1)^* - \widehat\phi(1)\widehat\phi(j)^*=\mu^{-1}b_j(\widehat\phi(j),\widehat\phi(1))j=\mu^{-1}b_j(j,1)j=2j,\]
    implying that $\iota(\phi)$ stabilizes $1_-$. Thus we have the inclusion morphism $\bInv(Q)\to\bStab(1_+,\mu^{-1}1_-)$. We want to show that it is an isomorphism. 
    
    Take $\psi\in\bStab(1_+,\mu^{-1}1_-)(S)$. Then $\psi([1_+,\mu^{-1}1_-])=[1_+,\mu^{-1}1_-]$, which is a grading element, so by Remark \ref{Rderivation}, $\psi\in\bAutgr(L)$. In other words $\psi=\iota(\phi)$ for some $\phi\in\bStr(B)$. Since $\iota(\phi)$ stabilizes $1_+$, we have
\[\phi(j)\phi(1)^* - \phi(1)\phi(j)^*=2\iota(\phi)(j)=2j\]
and multiplying through by $j$ we obtain
\[b_j(\phi(j),\phi(1))=2\mu.\]
Then Proposition \ref{PInv}(iii) yields $\phi\in\bInv(Q)(S)$. This completes the proof.
\end{proof}

\begin{Thm}\label{Tquotient} The fppf-quotient $\bG/\bInv(Q)$ is represented by a smooth affine scheme, and the map $\Pi:\bG\to\bY^L$ defined by $\phi\mapsto(\phi(1_+),\phi(\mu^{-1}1_-))$ for all $\phi\in\bG(S)$ with $S\in\Ralg$, induces an isomorphism $\bG/\bInv(Q)\overset{\sim}{\longrightarrow} \bY^L$.
\end{Thm}

\begin{proof} By Proposition \ref{Smoothaut} and Lemma \ref{Smooth}, both $\bG$ and $\bInv(Q)$ are smooth, and by Proposition \ref{Pstab}, $\bInv(Q)$ is the stabilizer of an $R$-point of $\bY^L$. Then \cite[XVI.2.2]{SGA3} and \cite[VIB.92(xii)]{SGA3}, taken together, imply that the quotient is representable by a smooth scheme that is locally of finite presentation, and that the induced map $i:\bG/\bInv(Q)\to \bY^L$ is a monomorphism. Since $\bY^L$ is also locally of finite presentation, to show that $i$ is an isomorphism it is enough, by  \cite[$_4$.17.9.5]{EGAIV}, to prove that $i$ becomes an isomorphism over any algebraically closed field $k\in\Ralg$. Over $k$, $\bY^L$ is smooth by Corollary \ref{Csmooth}, and the fact that $i$ is an isomorphism follows from Proposition \ref{Trans} by \cite[III.3.2.1]{DG}.
\end{proof}

\section{Torsors and twists}
We continue assuming that $L=L(B)$ where $B^-$ contains $j$ with $j^2=\mu1_B\in R^*1_B$, and denoting $j$ by $1_\pm$ when viewed as an element of $L_{\pm 2}$. Theorem \ref{Tquotient} provides a quotient projection $\Pi:\bG\to\bY^L$ that is fppf-locally surjective.  Let $e=(e_+,e_-)\in\bY^L(R)$ and let $\bE_e$ be the inverse image of $e$ under $\Pi$, so that for any $S\in\Ralg$,
\[\bE_e(S)=\{g\in \bG(S) \mid \Pi(S)(g)=e_S:=({e_+}_S,{e_-}_S)\}.\]
Then Theorem \ref{Tquotient} implies that  for each $e\in\bY^L(R)$, $\bE_{e}$ is an fppf-torsor under $\bInv(Q)$. We want to understand the twists of $L(B)$ by these torsors. As it will turn out, these correspond to different gradings on $L$. The following proposition gives the basic relation.

\begin{Prp} Let $e=(e_+,e_-)\in\bY^L(R)$, set $D=\ad_{[e_+,e_-]}$ and  \[L_i^D=\{x\in L\mid D(x)=ix\}.\]
Then this is a TKK-grading of $L(B)$, such that the $\pm 2$-component contains $e_\pm$.
\end{Prp}

\begin{proof}  By the Leibniz rule, if $x\in L_i^D$ and $y\in L_j^D$, then $[x,y]\in L_{i+j}^D$. Next we show that $L_i^D$ is finitely generated projective of constant rank $d_i$ (where $d_i$ is as in Definition \ref{DTKK}). These properties are invariant under faithfully flat descent. By Theorem \ref{Tquotient}, there is an fppf-ring extension $S$ of $R$ such that the action of $\bG(S)$ on $\bY(S)$ is transitive. Thus there is $\phi\in\bG(S)$ mapping $(e_+,e_-)$ to $(1_+,\mu^{-1}1_-)$. Setting $D'=\ad_{[1_+,\mu^{-1}1_-]}$, we have $D_S'\phi(x)=\phi D_S(x)$. Thus for $x\in L_S$, $D_S(x)=ix$ if and only if $D_S'(\phi(x))=i\phi(x)$, so $\phi$ induces a linear isomorphism $(L_S)_i^{D_S}\to (L_S)_i^{D'_S}$. But $D'$ is the grading derivation of the standard grading on $L$, so Lemma \ref{Lderivation} gives $(L_S)_i^{D'_S}=(L_S)_i$. Since $(L_S)_i$ is finitely generated projective  of constant rank $d_i$, so is $(L_S)_i^{D_S}=(L_i^D)_S$. 

In particular, this implies that $L_i^D=0$ if $|i|>2$. From this follows that $L_i^D\cap L_j^D=0$ if $i\neq j$, since then $(i-j)x=0$ is impossible for $x\neq 0$. Thus we have an inclusion $\bigoplus_{|i|\leq2} L_i^D\to L$. We claim that this is an equality: passing to the fppf-extension $S$ above and composing with the map $\phi$ yields equality over $S$, from which the claim follows by descent.

Thus $L=\bigoplus_{|i|\leq2} L_i^D$ is a TKK-grading. For each $x\in L_i^D$ we have $D(x)\otimes 1_S=ix\otimes1_S$ with $S$ as above. By faithful flatness this implies that $D(x)=ix$. Thus $D$ is the grading derivation. By definition of $\bY^L$ we get $D(e_\pm)=\pm2e_\pm$, which implies that $e_\pm\in L_{\pm2}^D$.    
\end{proof}

We will simply call the grading with grading derivation $\ad_{[e_+,e_-]}$ the \emph{$e$-grading} on $L(B)$. The Lie algebra $L(B)$  endowed with this grading is denoted by $L(B,e)$. Thus $L(B,e)=L(B)$ as an (ungraded) Lie algebra. If $e=(1_+,\mu^{-1}1_-)$, the $e$-grading on $L(B)$ is the standard grading, so that in this particular case, $L(B,e)=L(B)$ as a graded Lie algebra. We call this choice of $e$ the \emph{standard extremal pair} of $L(B)$. When we consider $L(B)$ as a graded Lie algebra without specifying the grading, we always mean the standard grading.

We now want to understand the gradings on those Lie algebras $L'$ that are isomorphic to $L(B)$.

\begin{Lma}\label{Lgriso} If $L'$ is a Lie algebra with $L'\simeq L(B)$, then $L'=L(B')$ for some Brown algebra $B'$.    
\end{Lma}

\begin{proof} Assume that $\phi:L(B)\to L'$ is an isomorphism of Lie algebras. Then $\phi$ transports the grading of $L(B)$ to $L'$ so that $(L')_i=\phi(L(B)_i)$, and $L'$, with this grading, is graded isomorphic to $L(B)$. The grading on $L'$ has grading element $\phi(V_{1,1})$. Since $V_{1,1}=[x_+,x_-]$, where $x_\pm\in L(B)_{\pm1}$ is the identity element of the Brown algebra $B$, we have $\phi(V_{1,1})=[\phi(x_+),\phi(x_-)]$ with $\phi(x_\pm)\in (L')_{\pm1}$. By \cite[Lemma 4.16]{Sta}, this implies that $L'=L(B')$ for a structurable algebra $B'$. By Lemma \ref{Lstructurable}, $B'$ is isotopic to $B$, and is thus a Brown algebra.
\end{proof}

\begin{Prp}\label{Pgriso} Let $\phi:L(B)\to L(B')$ be an isomorphism of TKK-algebras. Let $e=(e_+,e_-)\in \bY^{L(B)}(S)$ and set $e'=(\phi(e_+),\phi(e_-))$. Then $e'\in\bY^{L(B')}(S)$ and $\phi$ is an isomorphism of graded Lie algebras $L(B,e)\to L(B',e')$.
\end{Prp}

\begin{proof} If $e\in \bY^{L(B)}(S)$, then Remark \ref{Rexiso} implies that $\phi(e_\pm)\in \bX^{L(B')}$, and satisfy
\[r_{\phi(e_+)}(\phi(e_-))=r_{e_+}(e_-)=-2,\]
so $e'\in\bY^{L(B')}(S)$. Moreover, $\phi$ maps the grading element $[e_+,e_-]$ to $[\phi(e_+),\phi(e_-)]$, which is a grading element of $L(B',e')$. Since both gradings are TKK-gradings, they are determined by their grading elements according to Lemma \ref{Lderivation}, whence $\phi$ is an isomorphism of graded Lie algebras as claimed.
\end{proof}

\begin{Cor} Let $L'$ be a TKK-algebra of Brown type $\E_8$ over $R$. Then $L'$ is isomorphic to $L(B)$ as a Lie algebra if and only if $L'$ is isomorphic, as a graded Lie algebra, to $L(B,e)$ for some $e\in\bY^{L(B)}(R)$.
\end{Cor}

\begin{proof} The if-direction is clear since $L(B,e)=L(B)$ as a Lie algebra. Conversely, assume that $\phi:L'\to L(B)$ is an isomorphism of Lie algebras. Then $L'=L(B')$ by Lemma \ref{Lgriso}. Let $e'$ be the standard extremal pair of $L(B')$ and let $e$ be its image under $\phi$. Then by Proposition \ref{Pgriso}, we have $e\in\bY^{L(B)}(R)$ and $\phi$ is an isomorphism of graded Lie algebras $L'\to L(B,e)$.
\end{proof}

Thus to understand the graded structure of those $L'$ that are isomorphic to $L(B)$, it is enough to understand $L(B,e)$ for $e\in \bY(R)$. We will identify this as the twist of $L=L(B)$ by the torsor $\bE_e$ above. Let us first recall how such twists work.

Consider the assignment that to each $S\in\Ralg$ assigns the set of equivalence classes $(\bE_e(S)\times L_S)/\sim$, where $\sim$ is the equivalence relation defined by
\[(\phi,x)\sim (\phi',x') \quad \Longleftrightarrow\quad \exists \rho\in \bInv(Q)(S): (\phi\rho,x)=(\phi',\rho(x')).\]
Then this is a set-valued functor on $\Ralg$, hence a presheaf. We denote its associated fppf-sheaf by $\bE_e\wedge L$. This is a sheaf of algebras. By the universal property of associated sheaves, the algebra structure is determined by the algebra structure on $(\bE_e(S)\times L_S)/\sim$ for those $S$ such that $\bE_e(S)\neq\emptyset$. Note that $\bInv(Q)(S)$ acts transitively on each such $\bE_e(S)$, so we may fix $\phi\in\bE_e(S)$ and write every element of $\bE_e(S)$ as the equivalence class of $(\phi,x)$ for some $x\in L_S$. Using the notation $\overline{(\phi,x)}$  for the equivalence class, the addition, scaling and bracket are defined as
\[\overline{(\phi,x)}+\overline{(\phi,y)}=\overline{(\phi,x+y)},\quad s\overline{(\phi,x)}=\overline{(\phi,sx)} \quad \text{and}\quad [\overline{(\phi,x)},\overline{(\phi,y)}]=\overline{(\phi,[x,y])}\]
for all $x,y\in L_S$ and $s\in S$. Thus the bracket is antisymmetric and satisfies the Jacobi identity, since so does the bracket in $L$, so  this is a Lie algebra. It is $\Z$-graded by declaring, for each $i\in\Z$, that the degree of $\overline{(\phi,x)}$ is $i$ if and only if $\deg_{L}x=i$. This is well-defined since if $\overline{(\phi,x)}=\overline{(\phi',x')}$, then $x=\rho(x')$ for some $\rho\in\bInv(Q)(S)\subseteq\bAutgr(L)(S)$.

\begin{Prp}\label{Passign} For each $e\in\bY^L(R)$, the twist $\bE_e\wedge L(B)$ is canonically isomorphic, as a graded Lie algebra, to $L(B,e)$. Namely, for all $S\in\Ralg$ with $\bE_e(S)\neq\emptyset$, the map
\[\Psi_e(S): (\bE_e(S)\times L(B)_S)/\sim \to L(B,e)_S,\qquad \overline{(\phi,x)}\mapsto \phi(x),\]
is an isomorphism of graded Lie algebras and induces the desired isomorphism. 
\end{Prp}

\begin{proof} We need to prove that the map $\Psi_e(S)$ is an isomorphism of graded Lie algebras whenever $S$ is such that $\bE_e\neq\emptyset$. Given that, the result follows by the definition of associated sheaves. The map is well-defined by definition of $\sim$.

If $\phi\in \bE_e(S)$, then $\phi\in\bAut(L(B))$ and satisfies $(\phi(1_+),\phi(\mu^{-1}1_-))=e$. By Proposition \ref{Pgriso}, this implies that $\phi$ is an isomorphism of graded Lie algebras $L(B)\to L(B,e)$. Thus $\Psi_e$ maps $\overline{(\phi,x)}+\overline{(\phi,y)}=\overline{(\phi,x+y)}$ to
\[\phi(x+y)=\phi(x)+\phi(y)=\Psi_e(\overline{(\phi,x)})+\Psi_e(\overline{(\phi,y)}),\]
and $s\overline{(\phi,x)}=\overline{(\phi,sx)}$ to 
\[\phi(sx)=s\phi(x)=s\Psi_e(\overline{(\phi,x)}),\]
and $[\overline{(\phi,x)},\overline{(\phi,y)}]=\overline{(\phi,[x,y])}$
to
\[\phi([x,y])=[\phi(x),\phi(y)]=[\Psi_e(\overline{(\phi,x)}),\Psi_e(\overline{(\phi,y)})].\]
Thus it is a Lie algebra homomorphism, which is bijective by the bijectivity of $\phi$. The fact that it is an isomorphism of graded Lie algebras follows from the corresponding fact for $\phi$. This completes the proof.
\end{proof}

Thus the TKK-gradings on $L(B)$ correspond precisely to the twists of $L(B)$ by the torsors $\bE_e$, where $e$ ranges over all extremal pairs, i.e.\ elements of $\bY^L(R)$.

We finish by proving that non-trivial twists occur, and see what it means for the algebras. First, we outline the relations between various degrees of equivalence.

\begin{Prp} Let $B$ and $B'$ be Brown algebras over $R$ containing conjugate invertible skew elements, and consider the following three statements:
\begin{enumerate}[(i)]
    \item The Brown algebras $B$ and $B'$ are isotopic.
    \item The Lie algebras $L(B)$ and $L(B')$ are graded isomorphic.
    \item The Lie algebras $L(B)$ and $L(B')$ are isomorphic.
\end{enumerate}
Then $(i)\Leftrightarrow (ii)\Rightarrow (iii)$. If $R$ is a field, then $(ii)\Leftrightarrow (iii)$. 
\end{Prp}

\begin{proof} Consider the chain of maps
\[H^1_{\mathrm{fppf}}(R,\bStr(B))\to H^1_{\mathrm{fppf}}(R,\bAutgr(L(B)))\to H^1_{\mathrm{fppf}}(R,\bAut(L(B)))\]
induced by the inclusions $\bStr(B)\to\bAutgr(L(B))\to\bAut(L(B))$. In each cohomology, the Brown algebra $B$ gives rise to a cohomology class, and so does $B'$;  in each cohomology, the class arising from $B$ is the base point. Note that  $B$ and $B'$ are isotopic if and only if the classes in the leftmost cohomology coincide, and $L(B)$ and $L(B')$ are graded isomorphic if and only if the classes in the middle cohomology coincide, and isomorphic if and only if the classes in the rightmost cohomology coincide. It thus follows that $(i)\Rightarrow(ii)\Rightarrow(iii)$. Since the inclusion $\bStr(B)\to \bAutgr(L(B))$ is an isomorphism by Proposition \ref{Pstrgr}, we get $(ii)\Rightarrow(i)$. 

Assume that $R=K$ is a field and that $L(B')$ is isomorphic to $L(B)$. Then $L(B')$ is graded isomorphic to $L(B,e)$ for some $e\in\bY^{L(B)}(K)$ by Proposition \ref{Pgriso}. By Propositions \ref{Trans} and \ref{Pgriso}, this is in turn graded isomorphic to $L(B,e_\lambda)$, where $e_\lambda=(\lambda1_+,(\lambda\mu)^{-1}1_-)$, for some $\lambda\in K^*$. But a grading element of $L(B,e_\lambda)$ is $[\lambda1_+,(\lambda\mu)^{-1}1_-]=[1_+,\mu^{-1}1_-]$, which is a grading element of $L(B)$. Thus $L(B,e_\lambda)$ is graded isomorphic to $L(B)$. Therefore $(iii)\Rightarrow(ii)$ when $R$ is a field.
\end{proof}

Lastly we show that the implication $(iii)\Rightarrow (ii)$ fails when $R$ is not a field. First we prove that the torsor of Theorem \ref{Tquotient} is non-trivial,  using the technique of  \cite{Gil1}.

\begin{Prp}\label{Pnontriv} Let $B$ be the (split) Brown algebra over $\C$ with $j=\diag(1,-1)$. Then the fppf-torsor $\bG\to\bY^{L(B)}$ under $\bInv(Q(B,j))$ is non-trivial.
\end{Prp}

\begin{proof} Consider the Lie groups $H=\bInv(Q(B,j))(\C)$ and $G=\bG(\C)$, of types $\E_7$ and $\E_8$, respectively.   By \cite[Lemma 1.4]{Als1}, it is enough to prove that for some $n$, the homotopy group $\pi_n(H)$ is not a direct factor of $\pi_n(G)$. By Cartan decomposition, $H$ (resp.\ $G$) is the direct product of the real compact Lie group $H_c$ of type $\E_7$ (resp.\ the real compact Lie group $G_c$ of type $\E_8$) and a Euclidean space. Thus it suffices to prove that $\pi_n(H_c)$ is not a direct factor of $\pi_n(G_c)$. By \cite[Theorem V]{BS} we have $\pi_{11}(H_c)\simeq \Z$ and $\pi_{11}(G_c)=0$. The conclusion follows. 
\end{proof}

\begin{Cor} There exists a $\C$-ring $R$ such that $\Spec (R)$ is a smooth $\C$-variety and such that there is a Brown algebra $B$ over $R$ with $L(B)$ isomorphic to the split TKK-algebra, but not graded isomorphic to it. Equivalently, $L(B)$ is isomorphic to the split TKK-algebra without $B$ being an isotope of the split Brown algebra.
\end{Cor} 

\begin{proof} Let $B^s$ be the (split) Brown algebra over $\C$ with the usual skew-element $j=\diag(1,-1)$ (satisfying $j^2=1$) and set $Q^s=Q(B^s,j)$ and $L^s=L(B^s)$. By Corollary \ref{Csmooth}, $\bY^{L^s}$ is a smooth variety and we set $R=\C[\bY^{L^s}]$. The inclusions $\bInv(Q^s)\to\bAutgr(L^s)\to\bAut(L^s)$ induce maps
\[H^1_{\mathrm{fppf}}(R,\bInv(Q^s))\overset{f}{\rightarrow} H^1_{\mathrm{fppf}}(R,\bAutgr(L^s))\overset{g}{\rightarrow} H^1_{\mathrm{fppf}}(R,\bAut(L^s)).\]
Since the torsor of Proposition \ref{Pnontriv} is non-trivial, there is $e\in \bY^{L^s}(R)$ that is not in the orbit of $(1_+,1_-)$. Now, \cite[Proposition 2.4.3]{Gil2} implies that $\ker (g\circ f)$ is in bijection with the set of $\bAut(L^s)(R)$-orbits in $\bY^{L^s}(R)$. Thus under this bijection, the orbit of $e$ corresponds to a non-trivial element $\xi$ in $\ker (g\circ f)$. By Corollary \ref{Ctrivker} and the fact that $\bAutgr(L(B))\simeq \bStr(B)$, the kernel of $f$ is trivial. Thus $f(\xi)$ is a non-trivial element in $\ker g$, that corresponds to the graded isomorphism class of $\bE_e\wedge L(B)\simeq L(B,e)$. Thus $L(B,e)$ is not graded isomorphic to $L(B)$ despite $L(B,e)=L(B)$ as (ungraded) Lie algebras. The conclusion follows.
\end{proof}

\end{document}